\documentclass[12pt]{amsart}

\usepackage{color,graphicx,mathrsfs,eucal,tikz,amscd,amssymb}
\usetikzlibrary{calc}
\usetikzlibrary{through}
\def\d{\partial}

\newcommand{\XXX}{\mathbb{X}}
\newcommand{\RRR}{\mathbb{R}}
\newcommand{\TTT}{\mathbb{T}}
\newcommand{\ip}[1]{\langle {#1} \rangle}
\newcommand{\CC}{\mathcal{C}}
\newcommand{\utheta}{\ensuremath{\underbar{\ensuremath{\theta\!}}}\,}
\newcommand{\spn}{\mathrm{span}}
\newcommand{\vh}{\hat v}
\newcommand{\wh}{\hat w}
\newcommand{\xt}{\tilde x}

\newcommand{\dmsr}{\;dy_4 dy_3 dy_2 dy_1}
\newcommand{\Ec}{\ensuremath{\mathcal{E}^\mathrm{curl}}}
\newcommand{\Ed}{\ensuremath{\mathcal{E}^\mathrm{div}}}
\newcommand{\nhkcurl}[2]{\ensuremath{\left\| {#1} \right\|_{H^k(\mathrm{curl},{#2})}}}
\newcommand{\nhkdiv}[2]{\ensuremath{\left\| {#1} \right\|_{H^k(\mathrm{div},{#2})}}}
\newcommand{\TT}{\mathcal{T}}
\newcommand{\defn}{\;\stackrel{\mathrm{def}}{=}\;}
\newcommand{\dive}{\mathop\mathrm{div}}
\newcommand{\grad}{\ensuremath{\mathop{\mathrm{grad}}}}
\newcommand{\gradt}{\ensuremath{\mathop{\mathrm{grad}_\tau}}}
\newcommand{\curl}{{\ensuremath\mathop{\mathrm{curl}\,}}}
\newcommand{\Hdiv}[1]{H(\dive,#1)}
\newcommand{\hdiv}[2]{H(\dive,#1,#2)}
\newcommand{\hodiv}[2]{H_{0,{#1}}(\dive,#2)}
\newcommand{\hocurl}[2]{H_{0,{#1}}(\curl,#2)}
\newcommand{\hcurl}[2]{H(\curl,#1,#2)}

\newcommand{\trc}{\mathop{\mathrm{trc}}}
\newcommand{\trct}{\mathop{\mathrm{trc}_\tau}}
\newcommand{\trcn}{\mathop{\mathrm{trc}_n}}

\newcommand{\Clement}{{C}l\'ement}
\newcommand{\Schoberl}{Sch{\"o}berl}
\newcommand{\Nedelec}{N{\'{e}}d{\'{e}}lec } 
       
\newcommand{\om}{{\Omega}}
\newcommand{\gm}{{\Gamma}}

\newtheorem{theorem}{Theorem}[section]
\newtheorem{lemma}{Lemma}[section]

\theoremstyle{definition}

\theoremstyle{remark}
\newtheorem{remark}[theorem]{Remark}
\newtheorem{definition}[theorem]{Definition}
\newtheorem{assumption}{Assumption}

\begin{document}

\title[Expansion of a Lipschitz domain]
{Partial expansion of a Lipschitz domain and some applications}

\author{J. Gopalakrishnan}
\address{University of Florida, Department of
  Mathematics, Gainesville, FL 32611--8105}
\email{jayg@ufl.edu}

\author{W.~Qiu}
\address{Institute for Mathematics and its Applications, University of Minnesota, 207 Church Street~S.E.,Minneapolis, MN 55455}
\email{qiuxa001@ima.umn.edu}
\thanks{Corresponding author: Weifeng Qiu (qiuxa001@ima.umn.edu)}

\thanks{This work was supported in part by the NSF under
  DMS-1014817. The authors also gratefully acknowledge the support
  from the IMA (Minneapolis) during their 2010-11 annual program}

% 26B35  	Special properties of functions of several variables, 
%                   Hölder conditions, etc.
% 26B12  	Calculus of vector functions
% 26A16  	Lipschitz (Hölder) classes
% 52B10  	Three-dimensional polytopes
% 46E35  	Sobolev spaces and other spaces of "smooth'' 
%                   functions, embedding theorems, trace theorems
% 65L60  	Finite elements, Rayleigh-Ritz, Galerkin 
%                   and collocation methods


\date{}

\dedicatory{}

%    Abstract is required.
\begin{abstract}
  We show that a Lipschitz domain can be expanded solely near a part
  of its boundary, assuming that the part is enclosed by a piecewise
  $C^1$ curve. The expanded domain as well as the extended part are
  both Lipschitz.  We apply this result to prove a regular
  decomposition of standard vector Sobolev spaces with vanishing
  traces only on part of the boundary. Another application in the
  construction of low-regularity projectors into finite element spaces
  with partial boundary conditions is also indicated.
\end{abstract}


\maketitle

\section{Introduction}

Boundary value problems posed on non-smooth domains, particularly
polyhedral domains, are pervasive in computational mathematics. As
such, it is central to understand the properties of functions spaces
on such domains. Lipschitz regularity of the boundary of a
computational domain is often a standard assumption in such
studies. In this work we provide a theoretical tool for Lipschitz
domains which can be useful when working with function spaces
resulting from boundary value problems with mixed boundary conditions,
i.e., when part of a Lipschitz boundary is endowed with essential
boundary conditions, while the remainder has natural boundary
conditions. 

Suppose $\om$ is a three-dimensional Lipschitz domain.  We will show
that given a part $\gm \subsetneq \d\om$ of the boundary (satisfying
certain regularity assumptions), there is a larger Lipschitz domain
$\tilde \om$ which is obtained by extending $\om$ only near $\gm$. The
existence of this domain is proved constructively, by transporting
$\gm$ using a transversal vector field. A number of technical problems
need to be overcome for the proof.  We adapt several known
techniques~\cite{Grisv85,HofmaMitreTaylo07}, such as the construction
of a smooth transversal vector field in a neighborhood of Lipschitz
domains, and the equivalence between Lischitzness and uniform cone
property, to surmount the technicalities.

As an example of how to apply the result, we use the expanded domain 
to prove a decomposition result for two Sobolev spaces of vector
functions whose (tangential or normal) traces vanish only on a part of
the boundary.  The analogues of these decompositions for the case of
no boundary conditions have been known in the
literature~\cite{BirmaSolom90}. They are often called ``regular
decompositions''~\cite{HiptmLiZhou09}. Such decompositions have turned
out to be a valuable tool in proving convergence of numerical
algorithms. As another application of the domain expansion result, we
provide a missing detail in the construction of low-regularity bounded
cochain projectors (\Schoberl\ projectors) having partial boundary
conditions.

We begin by stating our geometrical assumptions in the next section. There
we will also state Theorem~\ref{thm:expand} on the existence of the
expanded domain. In Section~\ref{sec:application}, we discuss a few
applications. In Section~\ref{sec:proof} we prove
Theorem~\ref{thm:expand}.

\section{Partial expansion of a Lipschitz domain}   \label{sec:expan}

We consider a three-dimensional domain $\om$.  In this section, we
state our result on how a {\em part} $\gm \subsetneq \d\om$ of a
Lipschitz boundary can be transported outward maintaining Lipschitz
regularity, under suitable assumptions.  The precise assumptions on
the domain $\Omega$ and the part of its boundary $\Gamma$ will be
detailed below.

\subsection{Geometrical  assumptions}

We begin with some standard definitions, restated in an equivalent
form convenient for our purposes.  These definitions also establish
the notations that we will use later in proofs.

\begin{definition} \label{defn:Lip}
  Let $D$ be a nonempty, proper open subset of~$\mathbb{R}^3$. Fix
  $p\in\d D$.  We call $D$ a {\em locally Lipschitz domain} near $p$
  if there exists a open neighborhood $\CC_{r,h}$ of $p$ in
  $\mathbb{R}^3$, and new orthogonal coordinates $(x_1,x_2,x_3)$
  such that in the new coordinates, $p=(0,0,0)$, the neighborhood can
  be represented by
  \[
  \CC_{r,h}=\{(x_1,x_2,0)+t\hat u:\; (x_1,x_2)\in (-r,r)\times
  (-r,r) \text{ and } -h< t <h \},
  \]
  where the vector $\hat u=(0,0,1)$ in the local coordinates,
  and
\begin{align}
\CC_{r,h}\cap D 
&=\CC_{r,h}\cap\{(x_1,x_2,0)+t\hat{u}: \;
    x_1,x_2\in (-r,r),  \; t>\zeta(x_1,x_2)\},
\\
\CC_{r,h}\cap\partial D 
&=\CC_{r,h}\cap\{(x_1,x_2,0)+t\hat{u}:\;
   x_1,x_2\in (-r,r), \;   t=\zeta(x_1,x_2)\},
\\
\CC_{r,h}\cap\overline{D}^{c} 
&=\CC_{r,h}\cap\{(x_1,x_2,0)+t\hat{u}:\;
    x_1,x_2\in (-r,r),\; t<\zeta(x_1,x_2)  \},
\end{align}
for some Lipschitz function $\zeta:[-r,r]^2\rightarrow\mathbb{R}$ satisfying
\begin{equation}
\zeta(p)=0,\;
 \text{ and } \vert \zeta(x_1,x_2) \vert < h  \text{ if } 
 x_1,x_2          \in (-r,r).
\end{equation}  
We call $\CC_{r,h}$ a {\em coordinate box near $p$ in the $\hat
  u$-direction}. The boundary $\d D$ is then said to be a {\em
  Lipschitz hypograph near $p$ in the $\hat u$-direction}. A domain
$D$ which is locally a Lipschitz hypograph near every point on~$\d D$ is
simply called a {\em locally Lipschitz} domain.  We say that $D$ is a
{\em Lipschitz domain} if it is a locally Lipschitz domain and $\d D$
is compact.
\end{definition}

This is a standard definition, e.g., it is equivalent
to~\cite[Definition~$1.2.1.1$]{Grisv85} -- see
also~\cite[pp.~89]{McLea00} and~\cite{HofmaMitreTaylo07}. The next
definition allows us to talk about parts of the boundary which are
regular in a certain sense.

\begin{definition}[Piecewise $C^1$ dissection]
\label{defn:diss}
Suppose $D$ is a Lipschitz domain in $\mathbb{R}^3$, with the
accompanying notations above. Consider a disjoint union
\begin{equation}       \label{eq:dissect}
\partial D=\Gamma_1\cup\Pi\cup\Gamma_2,
\end{equation} 
where $\Gamma_1$ and $\Gamma_2$ are disjoint, nonempty, relatively
open subsets of $\partial D$, having $\Pi$ as their common boundary in
$\partial D$.  We call~\eqref{eq:dissect} a {\em piecewise $C^1$
  dissection of $\partial D$} if for any $p\in\partial D$, the
coordinate box $\CC_{r,h}$ near $p$, and the local coordinates
$(x_1,x_2,x_3)$, given by the Lipschitz regularity (see
Definition~\ref{defn:Lip}) are such that the three sets
$\Gamma_1\cap\CC_{r,h}$, $\Pi\cap\CC_{r,h}$ and
$\Gamma_2\cap\CC_{r,h}$ have the representations
\begin{align*}
\Gamma_1\cap\CC_{r,h} & = \{(x_1,x_2,x_3)\in\CC_{r,h}:\;
          x_3 = \zeta(x_1,\varrho(x_1)),\;
          x_2 < \varrho(x_1) \},
\\
\Pi\cap\CC_{r,h} & = \{(x_1,x_2,x_3)\in\CC_{r,h}:\;
           x_3 = \zeta(x_1,\varrho(x_1)),\; 
         x_2 = \varrho(x_1)      \},
\\
\Gamma_2\cap\CC_{r,h} & = \{(x_1,x_2,x_3)\in\CC_{r,h}:\;
         x_3 = \zeta(x_1,\varrho(x_1)), \; x_2 > \varrho(x_1)
           \},
\end{align*}
for some Lipschitz function $\varrho:[-r,r]\rightarrow\mathbb{R}$, and
additionally, the map 
\[
x_1\mapsto (x_1,\varrho(x_1),\zeta(x_1,\varrho(x_1)))
\]
from $[-r,r]$ into $\CC_{r,h}$ is $C^{1}$ on $[-r,r]$, except finitely
many points. These finitely many exceptional points (where $\Pi$ is
not $C^1$) will be enumerated as $p_1,\ldots, p_m$ (see
Figure~\ref{fig:dissec}).
\end{definition}

With the help of these definitions, we now place the following
assumptions on $D$, and the part of the boundary $\Gamma$ used in
imposing the mixed boundary condition.

\begin{assumption}   \label{asm:geom}
  Assume that $\om$ is a Lipschitz domain and $\Gamma \cup \Pi \cup
  (\d \om \setminus\bar{\Gamma})$ is a piecewise $C^1$ dissection of
  $\d \om$. 
  % Also assume that $\d \om$ is connected. 
\end{assumption}

A typical case in practical computations occurs when $\om$ is a
Lipschitz polyhedron and $\gm$ is formed by the union of a few faces
of the polyhedron. We have in mind boundary value problems where one
type of boundary condition is imposed on $\gm$, while another boundary
condition is imposed on the remainder of the boundary.

\begin{figure}
  \centering
  \begin{tikzpicture}
    \coordinate (c1) at (1,1);
    \coordinate (c2) at (3,1);
    \coordinate (c3) at (5,1);
    \coordinate (c4) at (4,2.1);
    \coordinate (c5) at (6.2,0.7);
    \coordinate  (a1) at (0,0);
    \coordinate  (a2) at (2.5,2);
    \coordinate  (a3) at (7,1);
    \coordinate  (a4) at (6,0);
    \coordinate  (b1) at (6.2,-1.2);

    \draw (a1) .. controls (c1) .. (a2);
    \draw (a1) .. controls (c2) and (c3) .. (a4);
    \draw (a2) .. controls (c4) .. (a3);
    \draw (a3) .. controls (c5) .. (a4);

    \coordinate (i1) at ($(a1)!0.6!(a2)$);
    \coordinate (i2) at ($(a2)!0.09!(a3)$);
    \coordinate (i3) at ($(i2)!0.5!(a3)$);
    \coordinate (i4) at ($(i1)!0.15!(i3)$);
    \coordinate (i5) at ($(i4)!0.4!(a1)$);    
   
    \coordinate (i12) at ($(i1)!0.5!(i2)$);
    \coordinate (ci12) at ($(i12)!0.3!(a2)$);
    \coordinate (i54) at ($(i5)!0.5!(i4)$);
    \coordinate (ci54) at ($(i54)!0.1!(a2)$);
    \coordinate (i53) at ($(i5)!0.5!(i3)$);
    \coordinate (ci53) at ($(i53)!0.6!(c4)$);
    \coordinate (cci53) at ($(i53)!0.3!(a2)$);
    \coordinate (i23) at ($(i2)!0.5!(i3)$);
    \coordinate (ci23) at ($(i23)!0.2!(c4)$);

    \draw (i1) .. controls (ci12) .. (i2) 
          (i2) .. controls (ci23) .. (i3)
          (i3) .. controls (ci53) and (cci53) .. (i5)
          (i5) .. controls (ci54) .. (i4)
          (i4) -- (i1);

   \node at ($(i3)!0.3!(i4)$) {\raisebox{-1em}{$\Pi$}};         
   \node[anchor=north west] at ($(i12)!0.75!(a2)$) {$\gm$};         
   \node[anchor=west] at (b1) {$\om$};

   \draw (a1) -- (b1) -- (a4);
   \draw (a3) -- (b1);
   
   % % annotations 
   % \draw (a1) circle (2pt) node [left] {$a_1$};
   % \draw (a2) circle (2pt) node [left] {$a_2$};
   % \draw (a3) circle (2pt) node [left] {$a_3$};    
   % \draw (a4) circle (2pt) node [left] {$a_4$};
   % \fill (c1) circle (2pt) node [left] {$c_1$};     
   % \fill (c2) circle (2pt) node [left] {$c_2$};
   % \fill (c3) circle (2pt) node [left] {$c_3$};
   % \fill (c4) circle (2pt) node [left] {$c_4$};
   % \fill (c5) circle (2pt) node [left] {$c_5$};
   \fill (i1) circle (2pt) node [above left] {$p_1$};     
   \fill (i2) circle (2pt) node [above left] {$p_2$};     
   \fill (i3) circle (2pt) node [right] {$p_3$};     
   \fill (i4) circle (2pt) node [right] {$p_5$};     
   \fill (i5) circle (2pt) node [below left] {$p_4$};         
   
  \end{tikzpicture}
  \caption{A piecewise $C^1$ dissection of the boundary $\d \om$.}
  \label{fig:dissec}
\end{figure}
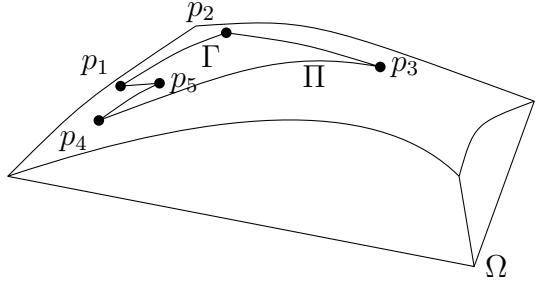

The next theorem shows that $\om$ can be expanded to a Lipschitz
domain in such a way that the expansion occurs only near $\gm$.

\begin{theorem}[Partial expansion of a Lipschitz domain] 
  \label{thm:expand}
  Suppose Assumption~\ref{asm:geom} holds. Then there exists a
  Lipschitz domain $\Omega^{e}$ such that
  \begin{equation*}
    \Omega\cap\Omega^{e} =\emptyset\qquad\text{ and }\qquad
    \partial\Omega\cap\partial\Omega^{e} = \Gamma \cup \Pi.
  \end{equation*}
  Furthermore, $\tilde\Omega = \Omega \cup \Gamma
  \cup \Omega^e$ is also Lipschitz.
\end{theorem}

The proof of this result is technical mainly because we cannot assume
more than Lipschitz regularity for $\d\om$. The  proof, together with
all the lemmas needed, are gathered in Section~\ref{sec:proof}.

\section{Applications}   \label{sec:application}

In this section, we give some applications of Theorem~\ref{thm:expand}
to questions in computational mathematics. In
\S~\ref{ssec:decomposition}, we use the first conclusion of
Theorem~\ref{thm:expand}, namely that $\om^e$ is Lipschitz, while in
\S~\ref{ssec:schob}, we use the second conclusion, namely that $\tilde
\om$ is also Lipschitz.  

Let us first establish notations for Sobolev spaces.  The set of
functions from $\Omega$ into $\XXX$ whose components are square
(Lebesgue) integrable will be denoted by $L^2(\Omega,\XXX)$, when
$\XXX$ is $\RRR,\RRR^3,\RRR^{3\times 3},$ etc. Let $H^1(\Omega,\RRR^3)
= \{ v \in L^2(\Omega,\RRR^3): \grad v \in L^2(\Omega, \RRR^{3 \times
  3})\}$, and $H^1(\Omega,\RRR)$ is similarly defined. Also, let
\[
\begin{aligned}
\hdiv \Omega {\RRR^3}
&  = \{ v\in L^2(\Omega,\RRR^3): \; \dive v \in
L^2(\Omega,\RRR)\},
\\
\hcurl \Omega {\RRR^3}
&  = \{ v\in L^2(\Omega,\RRR^3): \; \curl v \in
L^2(\Omega,\RRR^3)\}.
\end{aligned}
\]
Higher order analogues of this space are defined by 
\[
H^k(\curl,\Omega,\RRR^3) =
 \{v\in H^k(\Omega,\RRR^3): \curl v \in H^k(\Omega,\RRR^3)\},
\] 
for $k\ge 0$ (for the $k=0$ case, we obtain the previous space). The
space $H^k(\dive,\Omega,\RRR^3)$ is defined similarly.  Before we
proceed to the applications, let us review an interesting result for
the above defined space, recently obtained in~\cite{HiptmLiZhou09}.

\begin{theorem}[Hiptmair, Li, and Zhou~\cite{HiptmLiZhou09}]
\label{thm:steingen}
Let $\Omega$ be a Lipschitz domain and $k\ge 0$ be an integer. Then,
there are extension operators $\Ec: H^k(\curl,\Omega,\RRR^3) \mapsto
H^k(\curl,\RRR^3,\RRR^3)$ and $\Ed: H^k(\dive,\Omega,\RRR^3) \mapsto
H^k(\dive,\RRR^3,\RRR^3)$ and a $C>0$ (dependent on $k$) such that
\begin{align*}
 \nhkcurl{\Ec v}{\RRR^3}
 & \le C \nhkcurl{v}{\Omega},
 && \text{ for all } v \in H^k(\curl,\Omega,\RRR^3),
 \\
 \nhkdiv{\Ed u}{\RRR^3}
 & \le C \nhkdiv{u}{\Omega},
 && \text{ for all } u \in H^k(\dive,\Omega,\RRR^3),
\end{align*}
 where $ \nhkcurl{w}{D} = ( \| w \|_{H^k(D)}^2 + \| \curl w
  \|_{H^k(D)}^2 )^{1/2}$, and $\nhkdiv{w}{D}$ is  similarly defined.
\end{theorem}

This is a generalization of the Stein extension~\cite{Stein70} (and we
will use it in \S~\ref{ssec:decomposition} below).  The construction of
the above extensions are based on generalizing an integral formula of
Stein that extends functions on Lipschitz hypographs, in such a way
that a target commutativity property is satisfied. For any given~$k$,
a simpler such extension based on Hestenes' generalized
reflections~\cite{Heste41} can be constructed, as
in~\cite[\S~2.1]{DemkoGopalSchob09b}. However the result
of~\cite{HiptmLiZhou09} is stronger and gives a universal extension
for all~$k$, as stated above.

\subsection{A decomposition of spaces}   \label{ssec:decomposition}

As an application of Theorem~\ref{thm:expand} (and
Theorem~\ref{thm:steingen}) we now prove a decomposition of Sobolev
spaces that finds utility in analyses of certain computational
algorithms.

By way of preliminaries, recall~\cite{GirauRavia86} that the trace
operator and the normal trace operator, namely,
\[
\trc(v) = v|_{\d\Omega},
\quad\text{and}\quad
\trcn(v) = v\cdot n|_{\d\Omega},
\] 
resp., can be continuously extended from smooth vector functions to
$H^1(\om,\RRR^3)$ and $ \hdiv \Omega {\RRR^3}$, resp.  Throughout, $n$
denotes the outward unit normal on $\d\om$. Let 
\[
H_{0,\gm}^1(\om,\RRR^3) = \{ v \in H^1(\om,\RRR^3): \; \trc(v)|_{\gm} = 0\}. 
\]
Note that all components of a vector function in this space vanish on
$\gm$. Here, the range of $\trc(\cdot)$ is $H^{1/2}(\d\om)$, so the
restriction $\trc(v)|_\gm$ is obviously well-defined.  
We can also give meaning to a similar 
statement on the normal trace 
as follows.
The range of $\trcn$ as a map from $\hdiv\Omega{\RRR^3}$ equals
$H^{-1/2}(\d\Omega)$.  Let $\ip{\cdot,\cdot}_{H^{1/2}}$ be the duality
pairing between $H^{-1/2}(\d\Omega)$ and $H^{1/2}(\d \Omega)$, and let
$H^1_{0,\d\Omega\setminus\Gamma}(\Omega) = \{ z \in L^2(\Omega,\RRR): \;
\grad z \in L^2(\Omega,\RRR^3),\; \text{ and } z|_{\d\Omega \setminus
  \Gamma} =0 \}$.  For vector functions $v\in\hdiv\Omega{\RRR^3}$, we
say that
\begin{equation}
  \label{eq:4}
\trcn(v)|_{\Gamma}=0,   
\end{equation}
if $ \ip{ \trcn (v), \phi}_{H^{1/2}} = 0$ for all $\phi \in
H^1_{0,\d\Omega\setminus\Gamma}(\Omega,\RRR)$.  Define
\[
\hodiv \gm \om \defn \{ v \in \hdiv\om{\RRR^3} : \;  \trcn(v)|_{\Gamma}=0 
\}.
\]
Similarly we define
\[
\hocurl \gm \om 
\defn \{ v \in \hcurl\om{\RRR^3} : 
\ip{ n \times v, \trct(\phi)}_{H^{1/2}}=0 \text{ for all } 
\phi\in H^1_{0,\d\om\setminus \gm}(\om;\RRR^3)
\},
\]
where the tangential trace operator is defined by
\[
\trct (v) = \left( v - (v\cdot n)  n\right)|_{\d\om}
\]
for smooth functions~$v$. It is well known~\cite{GirauRavia86} that
$\trct$ can also be extended as a continuous map from $
\hcurl\om{\RRR^3}$ into $H^{-1/2}(\d\Omega,\TTT)$, where $\TTT$
  is the tangent space (homeomorphic to $\RRR^2$) and so we interpret $\trct(v)|_\gm$ just as
we did $\trcn(v)|_\gm$ in~\eqref{eq:4}.

\begin{theorem}
  \label{thm:decomp}
  Suppose $\om$ is contractible and Assumption~\ref{asm:geom}
  holds. Then, any $v\in \hodiv \gm\om$ can be decomposed into 
  \begin{equation}
    \label{eq:decomp1}
      v = \curl \varphi + \phi, \qquad \text{with}\quad
      \varphi \in H_{0,\gm}^1(\om, \RRR^3)
      \;\text{ and }
      \phi \in  H_{0,\gm}^1(\om, \RRR^3),
  \end{equation}
  and any $u$ in $\hocurl\gm\om$ can be decomposed into
  \begin{equation}
    \label{eq:decomp2}
      u = \grad \xi + \zeta, \qquad \text{with}\quad
      \xi \in H_{0,\gm}^1(\om, \RRR)
      \;\text{ and }
      \zeta \in  H_{0,\gm}^1(\om, \RRR^3).
  \end{equation}
  Moreover, both decompositions are stable, i.e., $\varphi$ and $\phi$
  depend continuously on $v$ (in their respective norms), and
  similarly $\xi$ and $\zeta$ depend continuously on~$u$.
\end{theorem}
\begin{proof}
  Let us prove the first decomposition for $v\in \hodiv \gm\om$.  Let
  $\Omega^e$ and $\tilde\Omega$ be as given by
  Theorem~\ref{thm:expand} and consider the trivial extension
  \begin{equation}
    \label{eq:vtilde}
     \tilde v = 
  \left\{
    \begin{aligned}
      v, & \text{ on }  \Omega, \\
      0, & \text{ on }  \tilde\Omega\setminus \Omega, 
    \end{aligned}
  \right.
  \end{equation}
  Clearly, $\tilde v$ is in $\hdiv{\tilde\Omega}{\RRR^3}$.
  
  First, we claim that there are functions $\bar\varphi$ and
  $\tilde\theta$, both in $ H^1(\tilde\Omega,\RRR^3)$, such that
  $\tilde v$ can be decomposed as
  \begin{equation}
    \label{eq:81}
    \tilde v = \curl \bar\varphi + \tilde\theta, \qquad\text{ on }\tilde\Omega.
  \end{equation}
  and the component functions $\bar\varphi,\tilde\theta$
  continuously depend on $\tilde v$. To see this, we first use a well known
  regular right inverse of divergence (see~\cite{Necas67},
  \cite[Corollary~2.4]{GirauRavia86}, or more recently~\cite{Bramb03}) to 
  obtain a $\tilde\theta \in H^1(\tilde \Omega,\RRR^3)$ satisfying
  $\dive\tilde\theta = \dive \tilde v$ and
  \begin{equation}
    \label{eq:101}
    \| \tilde\theta\|_{H^1(\tilde\Omega)} 
    \le C \| \dive \tilde v\|_{L^2(\tilde\Omega)} = C \| \dive v \|_{L^2(\Omega)}.
  \end{equation}
  Next, since $\dive(\tilde\theta - \tilde v) = 0$,
  by~\cite[Lemma~3.5]{AmrouBernaDauge98},
  there is a $\bar\varphi \in H^1(\tilde\Omega,\RRR^3)$ such that $
  \tilde\theta - \tilde v = \curl \bar\varphi$ (where we have
  used the contractibility of $\Omega$) and
  \begin{equation}
    \label{eq:111}
    \| \bar\varphi \|_{H^1(\tilde\Omega)}
    \le C \| \tilde\theta - \tilde v\|_{L^2(\tilde\Omega)}
    \le C \| v\|_{\Hdiv\om}.
  \end{equation}
  This
  proves~\eqref{eq:81}.

  Next, observe that when~\eqref{eq:81} is restricted to $\Omega^e$,
  since $\tilde v$ vanishes on $\Omega^e$, we have 
  \begin{equation}
    \label{eq:19}
    (\curl\bar\varphi)\big|_{\Omega^e} =
    -\tilde\theta\big|_{\Omega^e}\; \in H^1(\om^e,\RRR^3).
  \end{equation}
  Hence
  \[
  \bar\varphi|_{\Omega^e}\in H^1(\curl,\Omega^e,\RRR^3).
  \] 
  By Theorem~\ref{thm:expand}, $\om^e$ is Lipschitz, so we can apply
  the universal extension of Theorem~\ref{thm:steingen} to
  $\bar\varphi|_{\Omega^e}$, yielding $\hat\varphi = \Ec \bar\varphi$
  in $H^1(\curl,\RRR^3,\RRR^3)$. Adding and subtracting $\curl
  \hat\varphi$ in~\eqref{eq:81}, we thus have $ \tilde v = \curl (
  \bar\varphi - \hat\varphi) + (\tilde\theta + \curl\hat\varphi).$ In
  other words,
  \begin{equation}
    \label{eq:9}
    \tilde v = \curl \tilde \varphi
    + \tilde\phi,
    \qquad\text{ on } \tilde\Omega,
  \end{equation}
  with $\tilde \varphi = \bar\varphi - \hat \varphi$, and $\tilde\phi
  = \tilde\theta + \curl\hat\varphi$.

  To finish the proof of~\eqref{eq:decomp1}, we now only need to
  restrict the functions in~\eqref{eq:9} to $\om$.  Indeed, with
  $\varphi = \tilde\varphi|_\om$ and $\phi = \tilde\phi|_\om$, we have
  $v = \curl \tilde \varphi + \tilde\phi$.  We need to verify the
  boundary conditions of $\varphi$ and $\phi$.  To show that
  $\varphi|_{\Gamma}=0$, we only need to observe that $ \tilde
  \varphi|_{\om^e} = (\bar\varphi - \hat \varphi)|_{\om^e} = 0$
  because $\hat\varphi$ is the extension of $\tilde\phi$ from
  $\Omega^e$.  We note that {\em all} components of $\varphi$ vanish
  on $\gm$.

  To verify that all components of $\phi$ also vanish on $\gm$, recall
  that $\tilde v |_{\om^e}=0$. Since we observed above that $\tilde
  \varphi|_{\om^e} =0$, all the terms other than $\tilde\phi$
  in~\eqref{eq:9} vanish on $\Omega^e$, so $\tilde \phi$ must vanish
  on $\om^e$ too, and consequently, $\phi|_{\Gamma}=0.$ 

  It only remains to prove that the decomposition is stable. For this,
  \begin{align*}
    \| \phi\|_{H^1(\Omega)}
    & \le \| \tilde\theta\|_{H^1(\Omega)}  + 
       \|\curl(\Ec\tilde\varphi)\|_{H^1(\Omega)}  
        \le \| v \|_{\Hdiv\om}
  \end{align*}
  by~\eqref{eq:101}, Theorem~\ref{thm:steingen}, and ~\eqref{eq:111}.
  Similarly, 
  \begin{align*}
    \| \varphi\|_{H^1(\om)} 
    & \le \| \bar\varphi \|_{H^1(\tilde\om)} 
      + \| \Ec\bar\varphi \|_{H^1(\tilde\om)} 
      \\
    & \le C( \| \bar\varphi \|_{H^1(\tilde\om)} +  
            \| \curl \bar\varphi\|_{H^1(\om^e)})
            && \text{by Theorem~\ref{thm:steingen}}\\
     & \le  C( \| \bar\varphi \|_{H^1(\tilde\om)} +  
            \| \tilde\theta\|_{H^1(\om^e)})
            && \text{by~\eqref{eq:19}}
           \\
    & \le C \| v \|_{\Hdiv\om}
    && \text{by~\eqref{eq:101} and~\eqref{eq:111}}.
  \end{align*}
  Thus $\varphi$ and $\phi$ satisfy all the properties stated in the
  theorem.

  The proof of the other decomposition~\eqref{eq:decomp2} is similar.
\end{proof}

The topological assumption that $\om$ is contractible is used only to
convey the simplicity of the idea of the proof. It is possible to
prove a more general version of the theorem, accounting for nontrivial
harmonic forms.  We conclude with the following remarks on the
applications of the above decomposition.

\begin{remark}[{\em Overlapping Schwarz preconditioner}]
  In~\cite{PasciZhao02} we find a decomposition similar
  to~\eqref{eq:decomp2} but for the case of boundary conditions on the
  entire boundary (i.e., the case $\gm=\d\om$). This is a critical
  ingredient in their proof that additive and multiplicative
  overlapping Schwarz algorithms give uniform preconditioners for the
  inner product in $H_0(\curl,\om)$, even on non-convex domains. Other
  related works that paved the way for~\cite{PasciZhao02}
  include~\cite{ArnolFalkWinth97,HiptmTosel00,VassiWang92}.  In
  particular~\cite{ArnolFalkWinth97,HiptmTosel00} proved the
  uniformity of the preconditioner in the convex domain case. These
  results were used in~\cite{GopalPasci03} to prove that the
  overlapping Schwarz algorithms give a uniform preconditioner for the
  indefinite time-harmonic Maxwell equations by a perturbation
  argument. In view of Theorem~\ref{thm:decomp}, one can now extend
  the results of~\cite{PasciZhao02} and~\cite{GopalPasci03} to the
  case of {\em Maxwell equations with mixed boundary conditions} on
  general domains.
\end{remark}

\begin{remark}
We note that decompositions similar to~\eqref{eq:decomp2} were also
used in the analysis of the {\em singular field method}
in~\cite[Proposition~5.1]{BonneHazarLohre99}, but again only for the
case of boundary conditions on the entire boundary.  
\end{remark}

\begin{remark}
  Another application of Theorem~\ref{thm:decomp} is in the {\em
    characterization of traces} on Lipschitz boundaries. A Hodge
  decomposition of the space of tangential traces, namely $\trct(
  H(\curl,\om,\RRR^3))$, is already known~\cite{BuffaCostaSheen02}
  (and such results are useful in the analysis of boundary element
  methods for Maxwell equations).  Theorem~\ref{thm:decomp} gives a
  decomposition (different from the Hodge decomposition) of the traces
  into a regular and a singular part. Specifically, taking the
  tangential trace of~\eqref{eq:decomp2}, any $v_\tau \in \trct(
  \hocurl\gm\om)$ can be decomposed into
\begin{equation}
  \label{eq:20}
v_\tau = \gradt \xi + \zeta_\tau  
\end{equation}
where $\xi \in H_{0,\gm}^1(\om, \RRR)$ and $ \zeta \in
H_{0,\gm}^1(\om, \RRR^3).$ The $\zeta_\tau$ part is regular, while the
singular part is entirely a surface gradient. Moreover, both
components of the decomposition vanish on $\Gamma$. Such
decompositions were used (albeit on the surface of a tetrahedron)
in~\cite{DemkoGopalSchob09a,DemkoGopalSchob09b}. A decomposition of
normal traces analogous to~\eqref{eq:20}, but using a surface curl,
also follows from  Theorem~\ref{thm:decomp} (using~\eqref{eq:decomp1}).  
\end{remark}

\begin{remark}
  A {\em right inverse of the divergence} operator, with mixed
  boundary conditions, is provided by Theorem~\ref{thm:decomp}.  To
  see this, first note that given any $z\in L^2(\om, \RRR)$, it is
  easy to see that there exists a $v$ in $\hdiv\om{\RRR^3}$ satisfying
\[
      \dive v = z, \quad
      \trcn(v)|_{\Gamma} = 0, \quad \text{and}\quad
      \| v \|_{L^2(\Omega)} \le C \| z \|_{L^2(\Omega)}
\]
(consider the solution of $\Delta \psi = z$ with mixed boundary
conditions $ (\d \psi/\d n)|_{\Gamma} =0$ and $\psi|_{\d\om\setminus
  \Gamma} = 0$ and set $v = \grad \psi$). When this $v$ is decomposed
using~\eqref{eq:decomp1}, the resulting $\phi$ has all
  its components in $H^1_{0,\gm}(\om)$ and satisfies
\begin{equation}
  \label{eq:21}
  \dive \phi = z, \quad \phi|_{\gm}=0, \quad \text{and}\quad
  \| \phi \|_{H^1(\om)} \le C \| z \|_{L^2(\Omega)}.
\end{equation}
Thus, the map $z\mapsto \phi$ is a regular right inverse of
divergence.  (Note that~\eqref{eq:21} can also be proved by other
methods.) Right inverses of divergence are fundamental in the study of
{\em Stokes flow}~\cite{Ladyz63,ScottVogel85a}. The above result implies
that the Stokes system with no-slip conditions only on
$\gm$ is well posed. Another application of~\eqref{eq:21} is in
proving the well-posedness of mixed formulations of  {\em linear elasticity
with weakly imposed symmetry}.  Under purely traction boundary
conditions or purely kinematic boundary conditions, a proof of
well-posedness can be found in~\cite{Falk08}. The same applies almost
verbatim for mixed boundary conditions, once~\eqref{eq:21} is used, in
place of the right inverse of divergence used there.
\end{remark}

\subsection{\Schoberl\ projectors with partial boundary conditions}  
\label{ssec:schob}

Projectors from Sobolev spaces into finite element spaces with optimal
approximation properties find many applications in finite elements.
It is well known that every finite element has a canonical projector
defined by its degrees of freedom, but this projection is often
unbounded in the natural Sobolev space where the solution is
sought. This problem was first overcome by the \Clement\
interpolant~\cite{Cleme75}. Although \Clement\ interpolation
yielded operators bounded just in the $L^2$-norm, it had
neither the projection property, nor the commutativity with the
exterior derivative important in analyses of mixed methods.
\Clement's idea was substantially generalized by \Schoberl\
in~\cite{Schob01,Schob08c,Schob08b} to obtain similar projectors with
the additional commutativity properties.  These developments are
reviewed in~\cite{GopalOh10} where \Schoberl's ideas were generalized
to weighted norms.  We refer to the operators obtained by his method
as \Schoberl\ projectors.  The importance of these projectors have
also been highlighted in a recent review~\cite{ArnolFalkWinth10} of
finite element exterior calculus. They called the projectors bounded
cochain projectors, because the spaces formed a cochain
complex. Another recent work which refined \Schoberl's ideas
is~\cite{ChrisWinth08}, where the operators were called smoothed
projectors.

However, all these recent works dealt either with the case of no
boundary conditions or the case of homogeneous boundary conditions on
the entire boundary.  With the help of Theorem~\ref{thm:expand}, it is
easy to generalize their arguments to obtain a \Schoberl\ projector
with partial boundary conditions (only on $\gm$). (Actually
in~\cite{Schob08b}, the partial boundary condition case is considered
under the tacit unverified assumption that a result like
Theorem~\ref{thm:expand} holds.)  We now, very briefly, discuss the
case of the projectors with vanishing traces on $\gm \subsetneq
\d\om$.

Let $\om$ be a polyhedron satisfying Assumption~\ref{asm:geom}, meshed
by a geometrically conforming tetrahedral finite element
mesh~$\TT_h$. Apply Theorem~\ref{thm:expand} to obtain the associated
$\om^e$ and $\tilde \om$. Assume that $\TT_h$ is quasiuniform of mesh
size~$h$. Corresponding to each mesh vertex $x$, we associate a
ball~$\omega_x$ of radius~$h\delta$, where $\delta>0$ is a
global parameter to be chosen shortly. For vertices $x$ in $\bar \gm$,
we choose $\omega_x$ to be centered at some $\tilde x$ satisfying
$|\tilde x - x| \le c h \delta$ (where $c$ is another globally fixed
constant) and
\[
\omega_x \subset \om^e.
\]
For all other vertices~$x$, the ball $\omega_x$ is centered at $x$.

Now, given $u \in H_{0,\gm}^1(\om,\RRR)$, $v\in \hocurl\gm\om$, and
$w\in \hodiv\gm\om$, we extend each by zero to $\om^e$ to obtain
$\tilde u, \tilde v,$ and $\tilde w$ in $\tilde \om$. 
By Theorem~\ref{thm:expand}, $\tilde \om$ is a Lipschitz
domain. Therefore, we can use the universal extension of
Theorem~\ref{thm:steingen} to extend these functions to all
$\RRR^3$. Let us denote these extended functions on $\RRR^3$ by $\hat
u, \hat v$ and $\hat w$, resp. We will also consider a function $z \in
L^2(\om)$ and its trivial extension (by zero) $\hat z$ to
$L^2(\RRR^3)$.

Following~\cite{Schob08b}, we now define the smoothing operators.  Let
$K\in \TT_h$ and $x\in K$. Denote by $a_i$ the vertices of $K$.  and
by $\lambda_i(x)$ the barycentric coordinates at~$x$.  Define $ \tilde
x \equiv \tilde x (x, y_1, y_2, y_3, y_4) \equiv \sum_{i=1}^4
\lambda_i(x) y_i$. Let $\omega = \omega_{a_1} \times \omega_{a_2}
\times \omega_{a_3} \times \omega_{a_4}$ and abbreviate the (12
dimensional) measure on this product domain to $dy = \dmsr.$ Let $f_i$
denote a function in $L^\infty(\omega_{a_i})$ such that
$\int_{\omega_{a_i}} f_i(y_i) p(y_i) \; dy_i = p(a_i)$ for all
polynomials of some fixed degree.  Then, setting $\kappa \equiv
\kappa(y_1,y_2,y_3,y_4) \equiv f_1(y_1) f_2 (y_2) f_3(y_3) f_4(y_4)$,
define
\begin{align*}
   S^g u(x) &= \int_\omega \kappa \;\hat u(\tilde x ) \; dy
   \\
   S^c v(x) &= \int_\omega \kappa\; 
   \left(\dfrac{d \xt}{d x}\right)^T
   \hat v (\tilde x) \; dy, 
   \\
   S^d w(x) &= \int_\omega \kappa \;
      \det\left(\dfrac{d\xt}{d x}\right)
      \left(\dfrac{d\xt}{d x}\right)^{-T}
      \hat w(\tilde x)\; dy
      \\
   S^o z(x) &= \int_\omega \kappa \;
      \det\left(\dfrac{d\xt}{d x}\right)
      \hat z (\tilde x)\; dy
\end{align*} 
for all $ x\in K$ and for each $K \in \TT_h$. 

Next, let $I_h^g, I_h^c, I_h^d,$ and $I_h^o$ denote the canonical
interpolation operators of the lowest order Lagrange $(U_h)$,
\Nedelec\ $(V_h)$, Raviart-Thomas $(W_h)$, and $L^2$-conforming $(Z_h)$ 
finite element
spaces. The \Schoberl\ quasi-interpolation operators are now defined
by
\[
R_h^i = I_h^i \circ S^i, \qquad \text{ for } i \in \{ g,c,d, o\}.
\]
One can then prove, as indicated in~\cite{Schob08c} (or see more
details in~\cite[Lemma~4.2]{GopalOh10}), that the operators norms $\|
I - R_h^i \|_{L^2(\om)} = O(\delta)$. Hence, choosing $\delta$
sufficiently small, the operator $R_h^i$ restricted to the finite element
subspace is invertible. Let the inverse be $J_h^i$. The \Schoberl\
projectors are defined by
\[
\varPi_h^i = J_h^i \circ R_h^i.
\]
As in~\cite{Schob08c,Schob08b} (or cf.~\cite[Theorem~5.1]{GopalOh10}),
one can then continue on to prove that these projectors are continuous in
the $L^2(\om)$-norm, satisfy the commuting diagram
\begin{equation*}
  \begin{CD}
    H_{0,\gm}^{1}(\om)
    @>\grad >>  
    \hocurl\gm\om
    @>\curl >>
    \hodiv\gm\om
    @>\dive >>
    L^{2}(\om)
    \\
    @VV \varPi^g_h V               
    @VV \varPi^c_h V
    @VV \varPi^d_h V 
    @VV \varPi^o_h V 
    \\
    U_h
    @>\grad >>  
    V_h
    @>\curl >>
    W_h
    @>\dive >>
    Z_h,
  \end{CD}
\end{equation*}
and yield optimal approximation error estimates. This completes our
brief sketch of the construction of \Schoberl\ projectors for the partial
boundary condition case.

\section{Proof of Theorem~\ref{thm:expand}}   \label{sec:proof}

This section is devoted to proving Theorem~\ref{thm:expand}. The idea
is to construct the extended domain by ``transporting'' the boundary
$\gm$ outward along a continuous transversal vector
field. Technicalities arise when one makes change of variables and 
exhibits coordinate directions with respect to which the 
protruded boundary
is a Lipschitz hypograph.

Let $\d D$ be a Lipschitz hypograph near $p$ in the $\hat
u$-direction, and let $\CC_{r,h}$, $\zeta$ and $r$ be as in
Definition~\ref{defn:Lip}.  We say that $M$ is the {\em Lipschitz
  constant} of the Lipschitz hypograph $\d D$ if
\begin{equation}
  \label{eq:10}
  | \zeta(x_1,x_2)  - \zeta(y_1,y_2) |\le M \| (x_1,x_2) - (y_1,y_2) \|_2
\end{equation}
for all $(x_1,x_2)$ and $(y_1,y_2)$ in $(-r,r)$, i.e., $M$ is the
Lipschitz constant of $\zeta$. (Above and throughout $\| \cdot\|_2$
denotes the Euclidean distance.) Let $\gamma_M$ denote the acute angle
such that 
\begin{equation}
  \label{eq:gammaM}
  \tan \gamma_M = M \qquad\qquad (\gamma_M < \frac{\pi }2).
\end{equation}
Suppose $A$ and $B$ are any two points on $\d D
\cap \CC_{r,h}$ (see Figure~\ref{fig:pert}).  Then, the line segment
$AB$ connecting them has slope bounded by $M$ (for all such $A$ and
$B$) if and only if~\eqref{eq:10} holds.

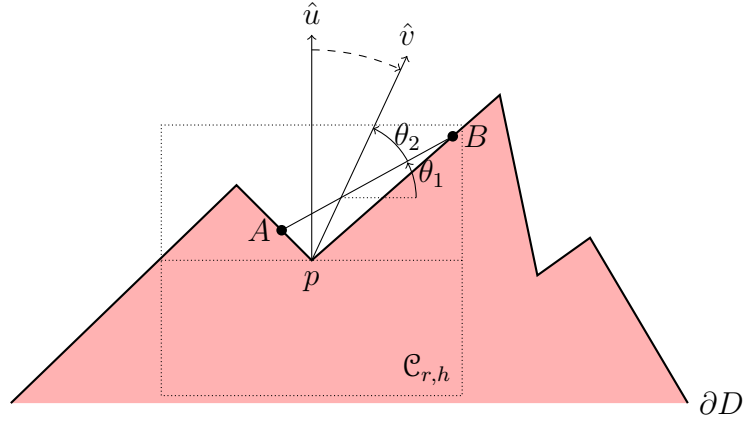
\begin{figure}
  \centering
  \begin{tikzpicture}
    \coordinate  (p) at (0,0);

    \coordinate  (p1) at (2.5,2.2);
    \coordinate  (p2) at (3,-0.2);
    \coordinate  (p3) at (3.7,0.3);
    \coordinate  (p4) at (5,-1.9);
    
    \coordinate  (q1) at (-1,1);
    \coordinate  (q2) at (-4,-1.9);
   
    \draw[thick,fill=red!30] (q2)--(q1)
             --(p) node [anchor=north] {$p$}
             --(p1)--(p2)--(p3)--(p4) node[anchor=west] {$\d D$};
  
    \coordinate [label=above:$\hat u$] (uh) at (0,3);
    \coordinate [label=above:$\hat v$] (vh) at (65:3);%(theta:r) specification
                                                      % polar coordinate 
    \draw[dashed,->] (0,2.8) arc (90:65:2.8);
    %\draw [->] (0.5,0)  node[anchor=south west] {$\theta$}  arc (0:65:0.5) ;

    \coordinate [label=left:$A$ ] (A) at ($(q1)!.6!(p)$);
    \coordinate [label=right:$B$] (B) at ($(p1)!.25!(p)$);
    \fill (A) circle (2pt);
    \fill (B) circle (2pt);

    \draw[->]  (p)--(uh);
    \draw[->]  (p)--(vh);
            
    \draw[densely dotted] (-2,0)--(2,0)   % coordinate box
                --(2,1.8)--(-2,1.8)
                --(-2,-1.8)--(2,-1.8)
                node [anchor=south east] {$\CC_{r,h}$} --(2,0);

    \draw (A)--(B);
    \coordinate (X)  at (intersection of A--B and p--vh);
    \coordinate (X1) at ($(X)+(1.0,0)$);
    \draw[densely dotted] (X)--(X1);
    \draw[->] (X1)  node[anchor=south] {\quad$\theta_1$} arc (0:28:1.0);
    \node (C) [circle through=(X1)] at (X) {};
    \coordinate (X2) at (intersection 1 of C and A--B);
    \draw[->] (X2) node[anchor=south] {$\theta_2$} arc (28:65:1.0) ;
  \end{tikzpicture}
  \caption{A 2D illustration.  Note that~\eqref{eq:10} is equivalent
    to $\theta_1 \le \gamma_M$.  % Also, in the 2D case,~\eqref{eq:11}
    % is equivalent to $\theta_1+\theta_2>\gamma_M$.
  }
  \label{fig:pert}
\end{figure}

\begin{lemma}[Perturbed coordinate direction]\label{lem:small_rotate}
  Suppose $\partial D$ is a Lipschitz hypograph near $p\in\partial D$
  in the $\hat{u}$-direction (and let $\gamma_M$ be as above). Let
  $\hat{v}$ be a unit vector such that
  \begin{equation}
    \label{eq:11}
     \sin\gamma_M   <\hat{u}\cdot\hat{v}.
  \end{equation}
  Then $\partial D$ is a Lipschitz hypograph near $p$ in the
  $\hat{v}$-direction.
\end{lemma}

\begin{proof}
  Let $\theta$ and $\theta_1$ denote the acute angles such that
  $\sin\theta = \hat u \cdot \vh$ and $\tan \theta_1$ equals the slope of
  $AB$, respectively (considering any two points $A$ and $B$ as
  mentioned above -- see also Figure~\ref{fig:pert}).  Let $\theta_2$ denote the smaller of the angles
  that $AB$ makes with $\vh$.  Then~\eqref{eq:11} implies $\gamma_M <
  \theta \le \theta_1 + \theta_2$ while~\eqref{eq:10} implies
  $\theta_1 \le \gamma_M$.  Hence $\theta_2 > \gamma_M - \theta_1 \ge
  0$, i.e., there is a $\utheta>0$ such that $\theta_2\ge \utheta$
  for all $A$ and $B$ in a neighborhood of $p$.  This implies that
  $\partial D$ is still a Lipschitz hypograph near $p$ in the
  direction of $\hat{v}$ with Lipschitz constant $\tan ( \pi/2
  - \utheta)$.
\end{proof}

Next, let us recall the well known fact that a domain is Lipschitz if
and only if it satisfies the uniform cone property, which we now
state.

\begin{definition}[Cone property]
\label{ConeProperty}
Let $D$ be an open subset of $\mathbb{R}^{3}$. We say that $D$ has the
{\em cone property at $p\in \d D$ in the direction $\hat u$} if there
are (i)~new coordinates $(y_1,y_2,y_3)$, where the $y_3$-direction is
$\hat u$, (ii)~a hypercube
\begin{equation*}
  V  =\{(y_{1},y_{2},y_{3}):-a_{j}\leq y_{j}\leq a_{j},1\leq j\leq 3\},
\end{equation*}
and~(iii)~constants $\theta\in (0,\pi/2)$ and $h>0$, and a
corresponding open cone $K_{\theta,h,\hat u}=\{y: (y_{1}^2 +
y_{2}^2)^{1/2} < y_3(\tan\theta)< h(\tan \theta) \}$, such that 
\[
y-z \in D\text{ whenever }
y \in \overline{D}\cap V
\text{ and }z \in K_{\theta,h,\hat u}.
\]
We say that $D$ has the \emph{uniform cone property} if every point
$p$ on $\d D$ satisfies the cone property (in some direction) with the
same $\theta$ and $h$.
\end{definition}

This definition can be found in~\cite[Definition~$1.2.2.1$]{Grisv85},
and so can the following theorem. The statement of the theorem
in~\cite[Theorem~1.2.2.2]{Grisv85} assumes boundedness
of~$D$. Following that proof we however find that boundedness of $D$
is unnecessary. We only need to assume that $\partial D$ is compact.

\begin{theorem}[see~Theorem~1.2.2.2 in~\cite{Grisv85}]
\label{thm:cone}
Let $D$ be an open subset of $\RRR^3$ with compact boundary $\d D$.
It has the cone property in $\hat u$-direction at some $p\in
\d D$ if and only if $\partial D$ is a Lipschitz hypograph in the
$\hat u$-direction near $p$. Moreover, $D$ is Lipschitz if and only if
$D$ has the uniform cone property.
\end{theorem}

% Recall that the ``direction'' of a Lipschitz hypograph is as in
% Definition~\ref{defn:Lip}.

\begin{lemma}[Combination of directions] \label{lem:spin}%
  Let $\nu$ denote the unit outward normal on $\d D$ at $p \in \d D$.
  If $\partial D$ is a Lipschitz hypograph near some $p\in \d D$ in
  both the directions $\hat{u}$ and $\hat{v}$, and
  \begin{equation}
    \label{eq:15}
      \hat{u}\cdot\nu>\kappa\quad\text{and}\quad\hat{v}\cdot\nu>\kappa,
  \quad\text{ for some } \kappa>0,
 \end{equation}
 a.e.~on a neighborhood of $p$ on $\partial D$, then
  $\partial D$ is also a Lipschitz hypograph near $p$ in the direction
  $\hat{w}=a\hat{u}+b\hat{v}$ for any $a,b>0$.
\end{lemma}
\begin{proof}
  We use Theorem~\ref{thm:cone}. By reducing, if necessary, the
  neighborhoods given by the cone property in $\hat u$ and $\hat v$
  directions, we can find a single hypercube $V$ containing $p$ such
  that for any $q\in\overline{D} \cap V$, we have $ q -
  K_{\theta,h,\hat u} \subset D$ and $ q - K_{\theta,h,\hat v} \subset
  D.$ Let $\CC_{r,h}$ be the coordinate box at $p$ in the $\hat
  u$-direction.  We choose $V$ so small that the sets $q -
  K_{\theta,h,\hat u}$, $q - K_{\theta,h,\hat v}$, and $q -
  K_{\theta,h,\hat w}$ are all contained in $\CC_{r,h}$ for all $q \in
  \bar D \cap V$. 
% The Lipschitz hypograph in the box $\CC_{r,h}$
%   naturally divides it into an ``inner part'' $\CC_{r,h}\cap D$ and a
%   remaining ``outer part''.

  We now claim that $D$ satisfies the cone property in the $\hat
  w$-direction at $p$.  If not, then by Definition~\ref{ConeProperty},
  for every neighborhood $V_n$, cone angle $\theta_n = \theta/n$, and
  cone height $h_n = h/n$, there is a point $y_n \in V_n
  \cap\overline{D}$ such that $y_n- K_{\theta_n,h_n,\hat w} \nsubseteq
  D$. Here we choose $V_n$ to be a hypercube of side-lengths $1/n$
  centered at $p$.  Thus, there exists $q_n$ in $y_n -
  K_{\theta_n,h_n,\hat w}$ such that
  \begin{equation}
    \label{eq:17}
    q_n \not\in D,
  \end{equation}
  and $q_n$ converges to $p$.

  Now, by~\eqref{eq:15}, $\hat{u}$ and $\hat{v}$ point above the
  hypograph in $\CC_{r,h}$, and since $a,b>0$,
  so does $\hat{w}$.  Since the heights and angles of the cones $y_n -
  K_{\theta_n,h_n,\hat w}$ approach $0$ and since their vertices
  approach $p$, we find that for sufficiently large $n$, the cone $y_n
  - K_{\theta_n,h_n,\hat w} \subseteq \CC_{r,h} \cap D$ and $q_n \in
  D$.  But this is in contradiction with~\eqref{eq:17}.
\end{proof}

\begin{lemma}[Separation of transported points]
\label{lem:inverse}
Let $u = (u_1,u_2,u_3)$. Suppose $u_j$ and $\zeta$ are Lipschitz
functions on $(-r,r)^2$ for some $r>0$,  $u_1(0)=u_2(0)=0$, 
and $u_3(0)\ne 0$. Let $(y_1,y_2)$ be mapped to
\[
L (y_1,y_2,s) = (y_1,y_2,\zeta(y_1,y_2)) + s \,u(y_1,y_2).
\]
Then there exists $0<r_{0}<r$ and $C>0$ (depending on $\zeta$ and $u$)
such that
\[
\|  L(y_1,y_2,s) - L(z_1,z_2,t)  \|_2 \ge C 
\| (y_1,y_2,s) - (z_1,z_2,t)  \|_2
\]
for all $(y_{1},y_{2},s)$ and $(z_{1},z_{2},t)$ in the cube
$(-r_{0},r_{0})^3$.
\end{lemma}

\begin{proof}
  Let $M$ be the maximum of the Lipschitz constants of $u_j$ and
  $\zeta$. Then, denoting by $[\cdot]_j$ the $j$th component, we have,
  for small enough $|s|$,
  \begin{subequations}
    \label{eq:1}
  \begin{align}
    \label{eq:L3}
    \big|
      \left[ L(y_1,y_2,s) - L(z_1,z_2,s) \right]_3
    \big|
    & \le\; 2 M 
    \| (y_{1},y_{2}) - (z_{1},z_{2}) \|_2,
    \\
    \label{eq:L12}
    \|   \ell(y_1,y_2,s) - \ell(z_1,z_2,s) \|_2
    & \ge
    (3/4)     \| (y_{1},y_{2}) - (z_{1},z_{2}) \|_2,
 \end{align}
  where $l(y_1,y_2,s) = ( [L(y_1,y_2,s)]_1, [L(y_1,y_2,s)]_2)$. 
  Let $c_0 = \vert u_3(0)\vert/16M>0$.
  Since 
  \[
  \| u(z_1,z_2) - (0,0,u_3(0) )\|_2^2 \le 3M^2\| (z_1,z_2)\|_2^2,
  \]
  for small enough $\| (z_1,z_2) \|_2$, we have
  \begin{align}
    \label{eq:u3z}
    \vert (s-t)\, u_{3}(z_{1},z_{2})\vert
    & \geq
    \frac 3 4 \vert u_{3}(0)\vert \, \vert s-t\vert,
    \\
    \label{eq:ujz}
    \big| (s-t)\, u_j(z_1,z_2)\big|
    & \leq
    c_0 \vert s-t\vert, \text{ for } j=1,\text{ and } 2,
  \end{align}
  for any $s,t\in (-r,r)$.
\end{subequations}
In view of these, there exists an $0<r_{0}<r$ such that all the
inequalities of~\eqref{eq:1} hold for any $(y_{1},y_{2},s)$ and
$(z_{1},z_{2},t)$ in $(-r_{0},r_{0})^{3}$.  The remainder of the proof
splits into two cases:

{\em Case~1:} $\vert s-t\vert\leq \| (y_1,y_2) -
(z_1,z_2)\|_2/(2\sqrt{2}c_0)$.  In this case,~\eqref{eq:ujz} implies
\begin{align}
\label{ineq_2_lemma_prime}
\vert (s-t)\, u_j(z_{1},z_2) \vert
\leq  \frac 1 {2\sqrt{2}}   \| (y_1,y_2) - (z_1,z_2) \|_2.
\end{align}
Then,
we have
\begin{align*}
  \| L(y_1,y_2,s) & - L(z_1,z_2,t)\|_2
  \ge
  \| l(y_1,y_2,s) - l(z_1,z_2,t)\|_2
  \\
  & 
  \ge 
  \| l(y_1,y_2,s) - l(z_1,z_2,s) \|_2
  - 
  \| l(z_1,z_2,s) - l(z_1,z_2,t) \|_2
  \\
  & 
  \ge
  ( \frac 3 4 - \frac 1 2 )
  \| (y_1,y_2) - (z_1,z_2) \|_2, 
\end{align*}
by~\eqref{eq:L12} and~\eqref{ineq_2_lemma_prime}.  This proves the
result in Case~1.

{\em Case~2:} $\vert s-t\vert \ge \| (y_1,y_2) - (z_1,z_2)\|_2
/(2\sqrt{2}c_0)$. We now estimate using the last component of $L$,
namely
\begin{align*}
   \| L(y_1,y_2,s) & - L(z_1,z_2,t)\|_2
  \ge
  \big| [ L(z_1,z_2,t) - L(y_1,y_2,s)]_3 \big|
  \\
  & \ge
   \big| [L(z_1,z_2,t) - L(z_1,z_2,s)]_3
   \big|
   -
   \big| [L(z_1,z_2,s) - L(y_1,y_2,s)]_3
   \big|
   \\
   & 
   =  |(s-t) u_3(z_1,z_2) | -    \big| [L(z_1,z_2,s) - L(y_1,y_2,s)]_3
   \big| 
\end{align*}
Now, by~\eqref{eq:L3},  and the inequality of Case~2, we have
\begin{align*}
\big| [L(z_1,z_2,s) - L(y_1,y_2,s)]_3   \big|
&
\le 2M     \| (y_{1},y_{2}) - (z_{1},z_{2}) \|_2
% \\
% &
\le  4  \sqrt{2} M c_0 |s-t|.
%= \frac{     \sqrt{2}   \vert u_3(0)\vert}{ 4} |s-t|
\end{align*}
Hence, using~\eqref{eq:u3z} and the definition of $c_0$, we obtain 
\begin{align*}
  \| L(y_1,y_2,s) & - L(z_1,z_2,t)\|_2
  \ge
  ( 
    \frac 3 4  - \frac{ \sqrt{2} }{4}
  )   | u_3(0)| \; |s-t|.
\end{align*}
Using the inequality of Case~2 to bound $|s-t|$ from below again, we
finish the proof.
\end{proof}

\begin{remark}
  \label{rem:leminvgeneral}
  Note that Lemma~\ref{lem:inverse} and its proof in fact holds more
  generally in $n+1$~space dimension, for any $n\ge 1$, for maps
  \[
  L (y_1,\ldots,y_n, s) = 
  (y_1,\ldots,y_n,\zeta(y_1,\ldots,y_n)) 
  +
  s \,u(y_1,\ldots,y_n)
  \]
  satisfying $u_j(0)=0$ for all $j=1,\ldots, n,$ but $u_{n+1}(0) \ne 0$.
  We only described it above for the $n=2$ case for simplicity. In the
  remainder of the paper, we will use it with $n=1$ and $2$.
\end{remark}

Now, let $D$ be a Lipschitz domain in $\mathbb{R}^3$.  Then there
exist finitely many coordinate boxes $\{\CC_{r_{j},h_{j}}\}_{j=1}^{N}$
such that
\begin{equation}
  \label{eq:3}
\partial
D\subset \mathop{\cup}_{j=1}^{N}\CC_{r_{j},h_{j}}.  
\end{equation}
Moreover, there
exists~\cite{McLea00} a partition of unity $\{\psi_{j}\}_{j=1}^{N}$
subordinate to $\{\CC_{r_{j},h_{j}}\}_{j=1}^N$ with
$\sum_{j=1}^{N}\psi_{j}=1$ on $\partial D$. We denote by $\nu_{j}$ the
direction of $\CC_{r_{j},h_{j}}$ for any $1\leq j\leq N$ (where the
``direction'' of the box is as in Definition~\ref{defn:Lip}).
We define a vector field 
\begin{subequations}
  \label{eq:2}
\begin{equation}
\hat{v}^{\prime}:=\sum_{j=1}^{N}\psi_{j}\nu_{j}.
\label{global_vector_field} 
\end{equation}
Then $\hat{v}^{\prime}$ is nonzero on $\partial D$.  Hence, we may
normalize and define
\begin{equation}
\hat{v}:=\frac{\hat{v}^{\prime}}{\vert\hat{v}^{\prime}\vert}.
\label{global_unit_vector_field} 
\end{equation}  
\end{subequations}
The following lemma proves that this yields a continuous transversal
vector field with an additional property we shall need later.  The
arguments are standard (see e.g.,~\cite{Grisv85,HofmaMitreTaylo07})
and we give the proof only for completeness.

\begin{lemma}[Transversal vector field]
\label{lem:transvers}
Let $D$ be Lipschitz.  The unit vector field $\hat v$ defined
by~\eqref{global_unit_vector_field} satisfies the following
properties:
\begin{enumerate}

\item It is tranversal, i.e., if $\nu$ is the outward unit normal on
  $\d D$, there is a constant $\kappa>0$ such that
  $\hat{v}\cdot\nu>\kappa$ a.e.~on $\d D$.

\item If $p_i$'s are the exceptional points in the piecewise $C^1$
  dissection (see Definition~\ref{defn:diss}), then
  $\hat v$ is constant on a neighborhood of each $p_i$.

\item \label{item:transverse3} In a neighborhood of any point $p\in \d
  D$, the boundary is a Lipschitz hypograph in the $\hat
  v(p)$-direction.
\end{enumerate}
\end{lemma}
\begin{proof}
  In each of the coordinate boxes in~\eqref{eq:3}, there is a
  $\kappa_j>0$ such that $\nu_{j}\cdot\nu>\kappa_{j}$. Set $\kappa =
  \min_j \kappa_j$. Then 
  \[
  \hat v' \cdot \nu = \sum_{j=1}^N \psi_j \nu_j \cdot \nu 
  \ge \kappa \sum_{j=1}^N \psi_j = \kappa
  \]
  a.e.~on $\d D$.

  To ensure that $\hat v$ is constant in the neighborhood of each
  $p_i$, we choose the covering of boxes in~\eqref{eq:3} as
  follows. We first select the coordinate boxes around each
  $p_i$. Next, we construct an open cover for the remainder such that
  the distances between $p_i$ and the open sets of this covering are
  bounded away from 0. A covering of $\d D$ is obtained by the union
  of this cover and the coordinate boxes of $p_i$.  Then we use a
  partition of unity subordinate to this union to
  define~\eqref{eq:3}. Then, only one of the summands
  in~\eqref{global_vector_field} is nonzero (say the $j$th) in a
  neighborhood of $p_i$ and $\psi_j$ is constant there.

  The final statement, item~(\ref{item:transverse3}), immediately
  follows from Lemma~\ref{lem:spin}.
\end{proof}

\begin{remark}
  \label{rem:LipschitzTrans}
  Although $\hat v'$ is obviously a smooth vector field in a
  three-dimensional neighborhood of $\d D$, below we will need to use
  the two-dimensional restriction $\breve v \equiv \hat v|_{\d D}$, which
  is {\em not} smooth, in general. Indeed, in each coordinate box where $\d
  D$ takes the form $(x_1,x_2, \zeta(x_1,x_2)),$ this vector field on
  $\d D$ is
  \[
  \breve v (x_1,x_2) = \hat v (x_1,x_2, \zeta(x_1,x_2) ),
  \]
  which only has Lipschitz regularity. To avoid proliferation of
  notations, we will avoid using $\breve v$ and continue to denote
  various restrictions of $\hat v$ by $\hat v$ itself.
\end{remark}
Using the above defined $\hat v$, we can transport the entire boundary
$\d D$ to create a new domain. Namely, define 
\begin{equation}
  \label{eq:5}
\Sigma_{t^{-},t^{+}}:=\{p+s\hat{v}(p):p\in\partial D,\;  t^{-}<s<t^{+}\},
\quad \text{ for any } t^{-}<t^{+}.  
\end{equation}

\begin{lemma}[Expansion of the entire boundary]
\label{lem:WholeExpansion}
Let $D$ be a Lipschitz domain and $\hat{v}$ be as above.
Then there exists $t_{0}>0$ such that:
\begin{enumerate}
\item \label{item:1}
  For any $p,q\in\partial D$ and any $s_{1},s_{2}\in [-t_0,t_0]$, 
  \[
  p+s_{1}\hat{v}(p)\neq
  q+s_{2}\hat{v}(q)\qquad\text{ whenever }
  (p,s_1)\neq (q,s_2).
  \]
  
\item \label{item:2} Let $-t_{0}\leq t^{-}<t^{+}\leq t_{0}$.  For any
  $q\in\partial\Sigma_{t^{-},t^{+}}$, the boundary
  $\partial\Sigma_{t^{-},t^{+}}$ is a Lipschitz hypograph near $q$ in
  the direction $\hat{v}(p)$.  Moreover,
  $\partial\Sigma_{t^{-},t^{+}}$ is compact. Consequently
  $\Sigma_{t^{-},t^{+}}$ is a Lipschitz domain.
\end{enumerate}
\end{lemma}
\begin{proof}
  Let $p \in \d D$ and let $\CC_{r,h}$ be the coordinate box in the
  $\hat v(p)$-direction near $p$. There is a Lipschitz function $\zeta$
  such that $\partial D$ around $p$ can be parameterized by the
  mapping $(x_{1},x_{2})\rightarrow (x_{1},x_{2},\zeta(x_{1},x_{2}))$
  in orthogonal coordinates $(x_{1},x_{2},x_{3})$.
 
  {\em Case~1:} We first show that the item~(\ref{item:1})
  holds, {\em assuming} that $q\in \d D$ is in the same coordinate box
  $\CC_{r,h}$ as $p$.  Then, since $\zeta$ and $\hat v$ are Lipschitz
  functions of $x_1,x_2$, we can apply Lemma~\ref{lem:inverse} (with
  $u$ there set to $\hat v$) to get that the distance between
  $p+s_{1}\hat{v}(p)$ and $q+s_{2}\hat{v}(q)$ is bounded below by $C
  \| (p,s_1) - (q,s_2) \|_2$. Hence this distance cannot be zero.

  {\em Case~2:} Now we prove the item~(\ref{item:1}) in general.
  If the result is not true, then for any $t^{-}<t^{+}$, there exists
  $p,q\in\partial D$ and $s,t\in [t^{-},t^{+}]$ such that
  $p+s\hat{v}(p)=q+t\hat{v}(q)$ and $p\neq q$. (If $p=q$, then we fall
  into Case~1 and the proof is done.)  This implies (choosing $t_\pm =
  \pm 1/i$), that there are
  $\{p_i\}_{i=1}^{\infty},\{q_i\}_{i=1}^{\infty}\in\partial D$ and
  $s_i,t_i\in [-1/i, 1/i]$ such that $p_i\neq q_i$ for any $i$ and
  $p_i+s_i\hat{v}(p_i)=q_i+t_i\hat{v}(q_i)$. Since $\partial D$ is
  compact, a subsequence of $\{(p_i,q_i)\}_{i=1}^{\infty}$
  converges. This implies (since $s_i,t_i\rightarrow 0$) that there
  exists an~$i$, large enough, such that $p_{i}$ and $q_i$ are in the
  same coordinate box.  Then, by Case~1, it is impossible that
  $p_i+s_i\hat{v}(p_i)=q_i+t_i\hat{v}(q_i)$, a contradiction.

  To prove the second item, we start by noting that by virtue of the
  first item, any point on $\d \Sigma_{t^{-},t^{+}}$ can be written
  (uniquely) as $p+ t_0 \hat v (p)$ for some $p \in \d D$. With
  $\CC_{r,h}$ and $ (x_{1},x_{2},\zeta(x_{1},x_{2}))$ as in the
  beginning of this proof, define new coordinates
  \[
  X(x_1,x_2) = 
  \begin{bmatrix}
    X_1 \\ X_2 
  \end{bmatrix}
  \equiv
  \begin{bmatrix}
    x_1 + t_0 \hat v_1 (x_1,x_2)\\
    x_2 + t_0 \hat v_2 (x_1,x_2)
  \end{bmatrix}
  \]
  Clearly, by choosing $t_0$ small enough, we can ensure that 
  \begin{equation}
    \label{eq:7}
      \| X(x_1,x_2) - X(y_1,y_2) \|_2 
      \ge C 
      \|(x_1,x_2) - (y_1,y_2) \|_2,
  \end{equation}
  so the inverse map $T = X^{-1}$ exists, i.e., $T(X_1,X_2) =
  (x_1,x_2)$.

  To prove that $\d \Sigma_{t^{-},t^{+}}$ is a Lipschitz hypograph
  near $p+ t_0 \hat v (p)$, we now only need to show that the third
  component $Z = \zeta(x_1,x_2) + t_0 \hat v_3 (x_1,x_2)$ is a
  Lipschitz function of the new coordinates $X_1$ and $X_2$. In these
  variables,
  \begin{equation}
    \label{eq:8}
      Z = \zeta \circ T + t_0 \hat v_3 \circ T.
  \end{equation}
  Since $\zeta$ and $\hat v_3$ are Lipschitz, and since $T$ is
  Lipschitz by~\eqref{eq:7}, we conclude from~\eqref{eq:8} that $Z$ is
  also Lipschitz, hence $\d \Sigma_{t^{-},t^{+}}$ is a Lipschitz
  hypograph in the $\vh$-direction.  Obviously, $\d \Sigma_{t_-,t_+}$
  is also compact.
\end{proof}

Next, given a $C^1$ dissection $\partial D =
\Gamma_1\cup\Pi\cup\Gamma_2$ of a Lipschitz boundary $\d D$, we
define, for $t\in \RRR$, the following sets of transported points
\begin{align*}
  S_t & :=\{p+t\hat{v}(p):\forall p\in\partial D\},
  &
  S_{1,t}& :=\{p+t\hat{v}(p):\forall p\in\Gamma_1\},
  \\
  \Pi_{t} & :=\{p+t\hat{v}(p):\forall p\in\Pi\},
  &
  S_{2,t} & :=\{p+t\hat{v}(p):\forall p\in\Gamma_2\}.
\end{align*}

\begin{lemma}[Transported dissection] \label{lem:ClosedCurve} 
  Let
  $\partial D = \Gamma_1\cup\Pi\cup\Gamma_2$ be a piecewise $C^{1}$
  dissection of a Lipschitz $\partial D$ and let $\hat{v}$ be
  transversal field of Lemma~\ref{lem:transvers}.  Then, there exists
  $t_0>0$ such that for any $-t_0\leq t\leq t_0$,
  $S_t=S_{1,t}\cup\Pi_{t}\cup S_{2,t}$ is a piecewise $C^{1}$
  dissection of $S_t$. The only points where $\Pi_t$ is not $C^{1}$ are
  $\{p_i+t\hat{v}(p_i)\}_{i=1}^{m}$. 
  % More precisely, for any
  % $p\in\Pi$, there is coordinate box $\CC_{r,h}$ near $p$ in the
  % direction $\hat{v}(p)$ such that $\Pi_{t}$ is piecewise $C^{1}$
  % around $p+t\hat{v}(p)$ with respect to $\CC_{r,h}$ for any
  % $-t_{0}\leq t\leq t_{0}$.
\end{lemma}

\begin{proof}
  By Lemma~\ref{lem:WholeExpansion}, there is a $t_0$ such that for all
  $t\in [-t_0,t_0]$, the surface $S_{t}$ is a Lipschitz hypograph at
  $p+t\hat{v}(p)$ in the $\hat{v}(p)$-direction for any $p\in\partial
  D$. 

  To show that $\Pi_t$ is piecewise $C^1$ dissection, we first
  consider the exceptional points $p_i$ of $\Pi$. The coordinate boxes
  near $p_i$ may simply be translated to form coordinate boxes around
  $p_i+t\hat{v}(p_i)$ because $\hat v$ is a constant vector field near
  $p_i$ (by Lemma~\ref{lem:transvers}). Since $\Pi_t$ is a merely a
  translation of $\Pi$, there is nothing to prove at these points.

  Next, consider the remaining $p \in \Pi$. In a neighborhood of
  $p + t \hat v (p)$, the transported curve $\Pi_t$ is $C^1$. This is
  because $\hat v|_\Pi$ is $C^1$ (and $\hat v$ is globally smooth), so
  $\Pi_t$ is locally the image of a $C^1$ curve under a $C^1$ map.
  Next, let $w$ denote the tangent vector of $\Pi_t$ at $p + t\hat
  v(p)$.  We construct a coordinate box around the $p+t\hat{v}(p)$ as
  follows: The $x_3$-direction is provided by $\hat v(p)$.  The
  $x_1$-direction is provided by the projection $\hat u = w - (w \cdot
  \hat v(p) ) \hat v(p)$ (and the $x_2$-direction is then
  determined). Then (because $\Pi_t$ can be locally parametrized using
  its tangent $w$) it is easy to see that $\Pi_t$ can be parametrized
  using a $C^1$ function of the new $x_1$~coordinate.
\end{proof}

Now, we are ready to define the partial expansion we need.  Let
$\partial D = \Gamma_1\cup\Pi\cup\Gamma_2$ be a piecewise $C^{1}$
dissection of a Lipschitz $\partial D$ and $\hat v$ be the vector
field of Lemma~\ref{lem:transvers}, define
\begin{equation}
  \label{eq:6}
  D_{t}^{e} =\{p+s\hat{v}(p):p\in\Gamma_1,s\in (0,t)\}.  
\end{equation}
Additionally let 
\[
\Psi_t =\{p+s\hat{v}(p):p\in\Pi,0<s<t\}, \qquad
\Gamma_{1,t} =\{p+t\hat{v}(p):p\in\Gamma_1\}.
\]
Clearly, for any $0<t\leq t_0$, $D_{t}^{e}$ is an open subset of
$\mathbb{R}^3$ and $\partial D_{t}^{e}
=\Gamma_1\cup\Gamma_{1,t}\cup\Psi_t\cup\Pi\cup\Pi_{t}$.

\begin{lemma}[The protrusion is Lipschitz]
  \label{lem:bulge}
  Let $D$ be a Lipschitz domain and let $\partial D =
  \Gamma_1\cup\Pi\cup\Gamma_2$ be a piecewise $C^{1}$
  dissection. Then, there exists $t_0>0$ such that for any $0<t\leq
  t_0$, the domain $D_{t}^{e}$ in~\eqref{eq:6} is a Lipschitz domain.
\end{lemma}
\begin{proof}
  We prove the Lipschtizness of each of the components in the
  decomposition $\partial D_{t}^{e}
  =\Gamma_1\cup\Gamma_{1,t}\cup\Psi_t\cup\Pi\cup\Pi_{t}$. 

  Obviously $\Gamma_1$, being part of $\d D$, is Lipschitz (since a
  Lipschitz function remains Lipschitz when the $x_3$-direction is
  reversed). That the surface $\Gamma_{1,t}$ is locally Lipschitz at
  all of its (interior) points follows from
  Lemma~\ref{lem:WholeExpansion}. Hence it suffices to consider points
  in the remaining components $\Psi_t, \Pi,$ and $\Pi_t$.

  To prove that $\Psi_t$ is a Lipschitz hypograph near a point $q \in
  \Psi_t$, we first note that $q\in \Pi_{t'}$ for some $0<t'<t$.  By
  Lemma~\ref{lem:ClosedCurve}, the curve $\Pi_{t'}$ is a piecewise
  $C^1$-dissection of $S_{t'}$. Then, reviewing the proof of
  Lemma~\ref{lem:ClosedCurve}, we find that there is a coordinate box
  $\CC_{r,h}$ near $q$ in the $\hat v(q)$-direction, where $\hat v$ is
  same transversal vector field given by Lemma~\ref{lem:transvers},
  such that $\Pi_{t'} \cap \CC_{r,h}$ in the local coordinates
  $(x_1,x_2,x_3)$ takes the form $(x_1, \rho(x_1),
  \zeta(x_1,\rho(x_1))$ (see Definition~\ref{defn:diss}). This means
  that near $q$, the surface $\Psi_t$ can be parametrized by
  \begin{equation}
    \label{eq:12}
      (x_1,s) \mapsto  
      X(x_1,s)\defn
      (x_1+ s\vh_1, \rho(x_1) + s \vh_2,  \zeta(x_1,\rho(x_1)) + s \vh_3)
  \end{equation}
  where $\vh_j$ is the $j$th component of $\vh( x_1, \rho(x_1), \zeta(
 x_1, \rho(x_1) )$.  Clearly, each $\vh_j$ is Lipschitz. We apply
 Lemma~\ref{lem:inverse} in $\spn(\vh_1,\vh_3)$ with $L = (X_1,X_3)
 \equiv X_{13}(x_1,s) \equiv (x_1,\zeta(x_1,\rho(x_1))) + s
 (\vh_1,\vh_3).$ (Note that the lemma is applicable in two-dimensions
 also -- see Remark~\ref{rem:leminvgeneral}.) Thus we obtain
 \begin{equation}
   \label{eq:14}
     \| X_{13}(x_1,s) - X_{13}(x_1',s') \|_2 \ge 
  C   \| (x_1,s) - (x_1',s') \|_2.
  \end{equation}
  Hence there is a Lipschitz inverse map $T = X_{13}^{-1}$ such that
  $T_1(X_1,X_3) = x_1$ and $T_2(X_1,X_3) = s$. We may therefore write
  $X_2$ in terms of $X_1$ and $X_3$, i.e., $X_2 = \rho(x_1) + s \vh_2
  = \rho \circ T_1 + T_2 \vh_2(T_1, \rho\circ T_1, \zeta(T_1,
  \rho\circ T_1)) \equiv Z(X_1,X_3)$. Clearly $Z$ is a Lipschitz
  function of $X_1$ and $X_3$.  Consequently, we may rewrite the
  surface representation in~\eqref{eq:12} using new independent
  variables $X_1$ and $X_3$ (keeping the same coordinate directions)
  as
  \[
  (X_1,X_3) \mapsto ( X_1,  Z(X_1,X_3), X_3 ).
  \]
  This proves that $\Psi_t$ is a Lipschitz hypograph near $q$ {\em in
    the $x_2$-direction}. (Without loss of generality, we choose the
  sign so that the $x_2$-direction points outward of $\Psi_t$.)

  It now only remains to consider points on $\Pi$ and $\Pi_t$. We will
  only consider $p \in \Pi$ as the other case is similar.
  We will use the fact that 
  \begin{equation}
    \label{eq:intersect}
    D_{t}^{e}=\Sigma_{0,t}\cap \Sigma_{-t,t}^{1}
  \end{equation}
  where 
  \begin{align*}
    \Sigma_{-t,t}^{1}& =\{q+s\hat{v}(q):q\in\Gamma_{1},s\in
    (-t,t)\}, 
    \\
    \Sigma_{0,t} & =\{q+s\hat{v}(q):q\in\partial D,s\in (0,t)\}.  
  \end{align*}
  Clearly $p \in \d \Sigma_{-t,t}^{1} \cap \d \Sigma_{0,t}$. We will
  prove that $\d D_t^e$ is a Lipschitz hypograph near $p\in \Pi$ in
  four steps:
  
  \begin{figure}
    \centering
    \begin{tikzpicture}
      \coordinate (a1) at (-0.7,0.1);
      \coordinate (a2) at (0,0);
      \coordinate (a3) at (1.5,-0.5);
      \coordinate (a4) at (2.5,0.25);
      \coordinate (a5) at (3,0.5);

      \coordinate (b1) at ($(a1)+(0,-1)$);
      \coordinate (b2) at ($(a2)+(0,-1)$);
      \coordinate (b3) at ($(a3)+(0,-1)$);
      \coordinate (b4) at ($(a4)+(0,-1)$);
      \coordinate (b5) at ($(a5)+(0,-1)$);

      \coordinate (t1) at ($(a1)+(0,1.4)$);
      \coordinate (t5) at ($(a5)+(0,1.0)$);

      \node [below left] at (t5) {$D$};
      \node [left] at (a1) {$\d D$};

      \fill [red!30,opacity=0.3]  (t1)--(b1)--(b2)--(b3)--(b4)--(b5)--(t5);
      \fill [red!30,opacity=0.5]  (a1)--(a2)--(a3)--(a4)--(a5)--(b5)--(b4)
                                --(b3)--(b2)--(b1)--cycle;

      \node [left] at ($(b5)!0.2!(a5)$) {$\Sigma_{0,t}$};

      \draw (a1)--(a2)--(a3)--(a4)--(a5);
      \draw (b1)--(b2)--(b3)--(b4)--(b5);
      \draw[line width=4pt, opacity=0.5] (a2)--(a3)--(a4);
      \node[above left] at (a3)  {\raisebox{0.5em}{$\gm_1$}} ;
      
      \coordinate (w) at ($(a2)+(-0.5,0.8)$);
      \coordinate (ww) at ($(a2)!0.6!(w)$);
      \coordinate (wl) at ($(ww)!0.2cm!90:(a2)$);
      \coordinate (wr) at ($(ww)!0.2cm!270:(a2)$);
      \draw[->] (a2)-- (w) node[above] {$\wh$};
      \fill [blue,opacity=0.5] (a2)--(wl)--(wr);
      
      \coordinate (wlt) at ($(wl)+(1.2,-1.5)$);
      \coordinate (wrt) at ($(wr)+(1.2,-1.5)$);
      \coordinate (wm)  at ($(wlt)!0.5!(wrt)$);
      \coordinate (pp)  at ($(wm)!0.6cm!270:(wrt)$);
      
      \fill (a2) circle (2pt) node [below left]  {$p$};
      \fill (pp) circle (2pt) node[left] {$p'$};
      \fill [blue, opacity=0.5] (pp)--(wlt)--(wrt)--cycle;

      % \draw[dotted] (a2)--(b2)  (a4)--(b4);
     
    \end{tikzpicture}
    \quad 
    \begin{tikzpicture}
      \coordinate (a1) at (-0.7,0.1);
      \coordinate (a2) at (0,0);
      \coordinate (a3) at (1.5,-0.5);
      \coordinate (a4) at (2.5,0.25);
      \coordinate (a5) at (3,0.5);

      \coordinate (b1) at ($(a1)+(0,-1)$);
      \coordinate (b2) at ($(a2)+(0,-1)$);
      \coordinate (b3) at ($(a3)+(0,-1)$);
      \coordinate (b4) at ($(a4)+(0,-1)$);
      \coordinate (b5) at ($(a5)+(0,-1)$);

      \coordinate (c1) at ($(a1)+(0,1)$);
      \coordinate (c2) at ($(a2)+(0,1)$);
      \coordinate (c3) at ($(a3)+(0,1)$);
      \coordinate (c4) at ($(a4)+(0,1)$);
      \coordinate (c5) at ($(a5)+(0,1)$);

      \coordinate (t1) at ($(a1)+(0,1.4)$);
      \coordinate (t5) at ($(a5)+(0,1.0)$);

      % \node [right] at (t5) {$D$};
      % \node [left] at (a1) {$\d D$};

      % D
      \fill [red!30,opacity=0.3] (t1)--(a1)--(a2)--(a3)--(a4)--(a5)--(t5);
      % Sigma
      \draw [fill=red!20] (c4)--(b4)--(b3)--(b2)
      --(c2)--(c3)--cycle;
      
      \node at ($(c3)!0.3!(a3)$) {$\Sigma^1_{-t,t}$};

      \draw (a1)--(a2)--(a3)--(a4)--(a5);
      \draw[line width=4pt, opacity=0.5] (a2)--(a3)--(a4);
      %\node[above] at (a3)  {$\gm$} ;
      
      \coordinate (w) at ($(a2)+(-0.5,0.8)$);
      \coordinate (ww) at ($(a2)!0.6!(w)$);
      \coordinate (wl) at ($(ww)!0.2cm!90:(a2)$);
      \coordinate (wr) at ($(ww)!0.2cm!270:(a2)$);
      \draw[->] (a2)-- (w) node[above] {$\wh$};
      \fill [blue,opacity=0.5] (a2)--(wl)--(wr);
      
      \coordinate (wlt) at ($(wl)+(1.2,-1.5)$);
      \coordinate (wrt) at ($(wr)+(1.2,-1.5)$);
      \coordinate (wm)  at ($(wlt)!0.5!(wrt)$);
      \coordinate (pp)  at ($(wm)!0.6cm!270:(wrt)$);
      
      \fill (a2) circle (2pt) node [below left]  {$p$};
      \fill (pp) circle (2pt) node[left] {$p'$};
      \fill [blue, opacity=0.5] (pp)--(wlt)--(wrt)--cycle;

      % \draw[dotted] (a2)--(b2)  (a4)--(b4);

    \end{tikzpicture}
    \quad
    \begin{tikzpicture}
      \coordinate (a1) at (-0.7,0.1);
      \coordinate (a2) at (0,0);
      \coordinate (a3) at (1.5,-0.5);
      \coordinate (a4) at (2.5,0.25);
      \coordinate (a5) at (3,0.5);

      \coordinate (b1) at ($(a1)+(0,-1)$);
      \coordinate (b2) at ($(a2)+(0,-1)$);
      \coordinate (b3) at ($(a3)+(0,-1)$);
      \coordinate (b4) at ($(a4)+(0,-1)$);
      \coordinate (b5) at ($(a5)+(0,-1)$);

      \coordinate (c1) at ($(a1)+(0,1)$);
      \coordinate (c2) at ($(a2)+(0,1)$);
      \coordinate (c3) at ($(a3)+(0,1)$);
      \coordinate (c4) at ($(a4)+(0,1)$);
      \coordinate (c5) at ($(a5)+(0,1)$);

      \coordinate (t1) at ($(a1)+(0,1.4)$);
      \coordinate (t5) at ($(a5)+(0,1.0)$);

      % \node [right] at (t5) {$D$};
      \node [above right] at (a5) {$\d D$};

      % % D
      % \fill [red!30,opacity=0.3] (t1)--(a1)--(a2)--(a3)--(a4)--(a5)--(t5);
      % D_t^e
      \draw [fill=red!30] (a4)--(b4)--(b3)--(b2)
      --(a2)--(a3)--cycle;
      
      \node at ($(b3)!0.5!(a3)$) {$D_t^e$};

      \draw (a1)--(a2)--(a3)--(a4)--(a5);
      \draw[line width=4pt, opacity=0.5] (a2)--(a3)--(a4);
      %\node[above] at (a3)  {$\gm$} ;
      
      \coordinate (w) at ($(a2)+(-0.5,0.8)$);
      \coordinate (ww) at ($(a2)!0.6!(w)$);
      \coordinate (wl) at ($(ww)!0.2cm!90:(a2)$);
      \coordinate (wr) at ($(ww)!0.2cm!270:(a2)$);
      \draw[->] (a2)-- (w) node[above] {$\wh$};
      \fill [blue,opacity=0.5] (a2)--(wl)--(wr);
      
      \coordinate (wlt) at ($(wl)+(1.2,-1.5)$);
      \coordinate (wrt) at ($(wr)+(1.2,-1.5)$);
      \coordinate (wm)  at ($(wlt)!0.5!(wrt)$);
      \coordinate (pp)  at ($(wm)!0.6cm!270:(wrt)$);
      
      \fill (a2) circle (2pt) node [below left]  {$p$};
      \fill (pp) circle (2pt) node[left] {$p'$};
      \fill [blue, opacity=0.5] (pp)--(wlt)--(wrt)--cycle;

      % \draw[dotted] (a2)--(b2)  (a4)--(b4);

    \end{tikzpicture}
    \caption{The uniform cone property in the $\wh$-direction (near
      $p\in \Pi$) holds for the domains $\Sigma_{0,t}$ (left) and
      $\Sigma_{-t,t}^1$ (center), so it holds for the
      intersection~$D_t^e$ (right).}
    \label{fig:intersect}
  \end{figure}
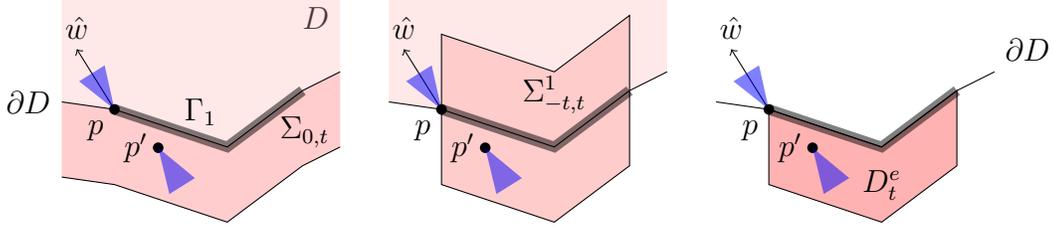

  {\em Step~1.} We claim that if there is a direction $\wh$ such that
  {\em both} $ \d\Sigma_{-t,t}^{1}$ and $ \d\Sigma_{0,t}$ are Lipschitz
  hypographs near $p\in \Pi$ in the $\wh$-direction, then the $D_t^e$ is
  also a Lipschitz hypograph near $p$ in the $\wh$-direction.

  To prove this claim, we use the uniform cone property. By
  Theorem~\ref{thm:cone}, there is a neighborhood $V_1$ of $p$ and a
  cone $K_{\theta_1,h_1,\wh}$ such that $p'- K_{\theta_1,h_1,\wh}
  \subseteq \Sigma_{-t,t}^{1}$ for all $p' \in V_1 \cap \bar
  \Sigma_{-t,t}^{1}$ (see Figure~\ref{fig:intersect}). Similarly there
  is a neighborhood $V_2$ and a cone $K_{\theta_2,h_2,\wh}$ such that
  $p'- K_{\theta_2,h_2,\wh} \subseteq \Sigma_{0,t}$ for all $p' \in
  V_2 \cap \bar \Sigma_{0,t}$. Hence considering a smaller hypercube
  $V$ in the intersection $V_1 \cap V_2$, in view
  of~\eqref{eq:intersect}, we find that for all $p' \in V \cap D_t^e$,
  we have $p'- K \subset D_t^e$, where $K$ is the smaller of the two
  cones. Thus $D_t^e$ satisfies the uniform cone property in the
  $\wh$-direction and the claim follows.

  {\em Step~2}. We now find a $\wh$ such that $ \d\Sigma_{0,t}$ is a
  Lipschitz hypograph near $p$ in the $\wh$-direction.

  We know that $\d D$ is a Lipschitz hypograph near $p$ in the $-\vh(p)$
  direction, which we now take to be our local $x_3$-direction -- see
  Figure~\ref{fig:wh}.  Let $\tan \gamma_M$ be the corresponding
  Lipschitz constant (as in~\eqref{eq:gammaM}). Define
 \begin{equation}
    \label{eq:13}
  \wh = (0,\cos\alpha,\sin\alpha), \qquad
  \text{where}
  \qquad
  \alpha = \frac 1 2 \left(\frac \pi 2 +  \gamma_M   \right).
  \end{equation}
  Clearly, by construction, $0<\gamma_M< \alpha < \pi/2$.
  Hence,~\eqref{eq:11} is satisfied since $\sin \gamma_M < \sin
  \alpha$. By Lemma~\ref{lem:small_rotate}, we conclude that
  $\d\Sigma_{0,t}$ is a Lipschitz hypograph near $p$ in the
  $\wh$-direction.

  \begin{figure}
    \centering
    \begin{tikzpicture}
      \coordinate (a1)  at (0,0);
      \coordinate (a2)  at (3,-1);
      \coordinate [label=above left:$\Pi$]  (a3)  at (4,0);
      \coordinate (a4)  at (0.5,0.5);
      
      \coordinate (b1) at (-1,-0.2);
      \coordinate [label=left:$\d D$] (b2) at (-2.3,1.5);
      \coordinate (b3) at (3.4,1.1);
      \coordinate (b4) at (4.7,2.);
      \coordinate (b5) at (0.5,2.3);
      \coordinate (b6) at (-1.2,2.0);
      
      \coordinate (c1) at ($(a1)-(0,1)$);
      \coordinate [label=below:$D_t^e$] (c2) at ($(a2)-(0,1)$);
      \coordinate (c3) at ($(a3)-(0,1)$);
      \coordinate (c4) at ($(a4)-(0,1)$);
      
      \draw[very thick,fill=red!50] (a1)--(a2)--(a3)--(a4)--cycle;
      \draw[fill=red!50] (a1)--(c1)--(c2)--(a2);
      \draw[fill=red!50] (a2)--(c2)--(c3)--(a3);
      
      \coordinate [label=above right:$p$] (p) at ($(a1)!0.4!(a2)$);
      \coordinate [label=above:$-\vh$] (p3) at ($(p)+(0,2.1)$);
      \coordinate [label=above:$x_1$] (p1) at ($(p)+(-1.3,.8)$);
      \coordinate [label=below left:$x_2$] (p2) at ($(p)+(-2,-0.7)$);
      \fill (p) circle (2pt);
      \draw[->] (p)--(p1) ;
      \draw[->] (p)--(p2) ;
      \draw[->] (p)--(p3) node [anchor=north west] {($x_3$)};
      
      \coordinate (t3) at ($(p)!0.15!(p3)$);
      \coordinate (t2) at ($(p)!0.15!(p2)$);
      \coordinate (tp3) at ($(t3)-(p)$);
      \coordinate (tp2) at ($(t2)-(p)$);
      \coordinate (t)  at ($(tp3)+(tp2)+(p)$);
      \coordinate (s1) at ($(p)!0.15!(p1)$);
      \coordinate (s2) at ($(p)!0.15!(p2)$);
      \coordinate (sp1) at ($(s1)-(p)$);
      \coordinate (sp2) at ($(s2)-(p)$);
      \coordinate (s)  at ($(sp1)+(sp2)+(p)$);
      \draw (p)--(t3)--(t)--(t2)--cycle;
      %\draw (p)--(s1)--(s)--(s2)--cycle;
  
      \coordinate (wh)  at ($1.6*(tp3)+6.65*(tp2)+(p)$);
      \draw[->] (p)--(wh) node [anchor=east] {$\wh$};
      \draw[->] ($(p)!0.65!(p2)$) arc (190:171:1) 
                 node [anchor=north east] {$\alpha$};

      \draw[fill=red!20,opacity=0.3] (a1)--(b1)--(a2);
      \draw[fill=red!30,opacity=0.3] (b4)--(a3)--(a4)--(b5);
      \draw[fill=red!30,opacity=0.3] (b6)--(a1)--(a4)--(b5);
      \draw[fill=red!30,opacity=0.3] (b2)--(b1)--(a1)--(b6);
      \draw[fill=red!30,opacity=0.3] (b3)--(a2)--(a3)--(b4);
      \draw[fill=red!30,opacity=0.3] (b3)--(a2)--(b1)--(b2);
    \end{tikzpicture}
    \caption{Determining the direction vector $(\wh)$ for the Lipschitz hypograph near $p\in \Pi$.}
    \label{fig:wh}
  \end{figure}
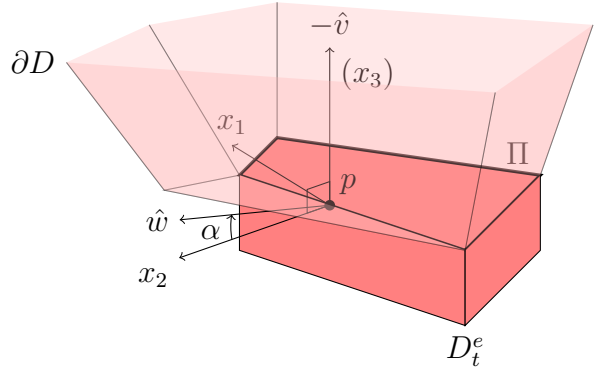

  {\em Step~3.} Suppose $p$ is not one of the exceptional points
  $p_i$. Then we claim that $\d\Sigma_{-t,t}^{1}$ is a Lipschitz
  hypograph near $p$ in the $\wh$-direction (with the same $\wh$ as
  in~\eqref{eq:13}).

  To prove this, we recall that $\d\Sigma_{-t,t}^{1}$ can locally be
  represented by the same map as in~\eqref{eq:12}.  However, this
  time, the components are not merely Lipschitz: they are in fact
  $C^1$. This is because, near $p$, the curve $( x_1, \rho(x_1),
  \zeta(x_1,\rho(x_1))$ is $C^1$, so each of the components must be
  $C^1$. Moreover, since $\vh_j$ is the $j$th component of the
  globally smooth $\vh$ along the $C^1$ curve, it is also $C^1$. Thus
  the map~\eqref{eq:12} defines a $C^1$ surface near $p$, so in
  particular, $\d\Sigma_{-t,t}^{1}$ is Lipschitz.  

  It only remains to verify that $\d\Sigma_{-t,t}^{1}$ is a Lipschitz
  hypograph {\em in the $\wh$-direction}. Standard geometrical
  arguments can be used to show this. But to be self-contained, we
  give a proof: The vector $\vh(p)$, which in local coordinates is
  $(0,0,1)$, together with the tangent to the curve $\Pi$ at $p$,
  which we denote by $(a,b,c)$, span the tangent plane of the $C^1$
  surface $\d\Sigma_{-t,t}^{1}$ at $p$. Moreover, due to the
  representation~\eqref{eq:12}, $a\ne 0$. Hence the normal vector $n$
  to $\d\Sigma_{-t,t}^{1}$ is in the direction $( 0,0,1) \times
  (a,b,c)$. In particular, $|n \cdot \wh |= |a \cos\alpha| \ne 0$.
  Now, since $\d\Sigma_{-t,t}^{1}$ is $C^1$, in the coordinate system
  $(z_1,z_2,z_3)$, with $p$ as the origin, and with $z_3$-direction
  equal to the $n$-direction, the surface $\d\Sigma_{-t,t}^{1}$ can be
  parametrized a $(z_1,z_2, \eta(z_1,z_2))$ for a $C^1$ function
  $\eta$ such that $\eta(z_1,z_2) = \lambda_1 z_1^2 + \lambda_2 z_2^2
  + o( \| (z_1,z_2)\|_2^2).$ This implies that the Lipschitz constant
  $M$ in these coordinates can be made arbitrarily small by
  considering a small enough neighborhood. Therefore we can apply
  Lemma~\ref{lem:small_rotate} (with $\hat u = n$ and $\hat v =
  \wh$). Since $ |n \cdot \wh |>0$, by choosing a small enough
  neighborhood, condition~\eqref{eq:11} can be satisfied. We conclude
  that $\d\Sigma_{-t,t}^{1}$ is a Lipschitz hypograph in the
  $\wh$-direction.

  {\em Step~4.} If $p$ coincides with one of the exceptional points
  $p_i$, then  $\d\Sigma_{-t,t}^{1}$ is a Lipschitz hypograph near $p$
  in the $\wh$-direction.

  To prove this, we recall that in a neighborhood of the exceptional
  points, the vector field $\vh$ is designed to be constant
  (Lemma~\ref{lem:transvers}). In the local coordinates, this constant
  vector is $\vh = (0,0,1)$. Hence the parametric representation of
  the surface $\d\Sigma_{-t,t}^{1}$ in~\eqref{eq:12} takes the form
  \[
  X(x_1,s) = (X_1,X_2,X_3) = 
  (x_1, \rho(x_1),  \zeta(x_1,\rho(x_1)) + s).
  \]
  Applying Lemma~\ref{lem:inverse} to $X_{13}(x_1,s) = (x_1,
  \zeta(x_1,\rho(x_1))) + s(0,1)$, we find, as before,
  that~\eqref{eq:14} holds.  Hence, as before, $X_{13}$ is invertible
  and the change of variable $(x_1,s) \mapsto (X_1,X_3)$ is locally
  one-one. Reparametrizing in new $(X_1,X_2,X_3)$-variables
  (but keeping the old coordinate directions) the surface takes the form
  \begin{equation}
    \label{eq:16}
     (X_1, \rho(X_1), X_3).
  \end{equation}
  Since $\rho$ is Lipschitz, this means that
  $\d\Sigma_{-t,t}^{1}$ is a Lipschitz hypograph near $p$ in the
  $x_2$-direction.

  This fact can be used to show that it is also a Lipschitz hypograph
  near $p$ in the $\wh$-direction. Denote the unit vectors in the
  $x_i$-coordinate direction by $e_i$. We rotate the coordinate
  directions in the $e_2$-$e_3$ plane, about the $e_1$-axis, by the
  angle $\alpha$, to get new orthogonal coordinate directions $e_i'$.
  Clearly, $e_2'=\wh$ due to~\eqref{eq:13}. If the coordinates $x_i'$
  are such that $x_1'e_1' + x_2'e_2' + x_3'e_3' = X_1e_1 + X_2e_2 +
  X_3e_3 $, then 
  \begin{equation*}
    \begin{bmatrix}
      x_1' \\ x_2' \\ x_3'
    \end{bmatrix}
    = 
    \begin{bmatrix}
      1 & 0 & 0 \\
      0 &  \phantom{-}\cos \alpha  & \sin\alpha\\
      0 &  -\sin \alpha & \cos\alpha
    \end{bmatrix}
    \begin{bmatrix}
      X_1 \\ X_2 \\ X_3
    \end{bmatrix}.
  \end{equation*}
  Using this to map~\eqref{eq:16}, we conclude that the points on
  surface $\d\Sigma_{-t,t}^{1}$ near $p$ takes the form $x_1'e_1' +
  x_2'e_2' + x_3'e_3'$ with 
  \[
  x_1'= X_1,\qquad
  x_2'= \rho(X_1)\cos \alpha + X_3\sin\alpha, \quad
  x_3'= -\rho(X_1) \sin\alpha+ X_3\cos\alpha.
  \]
  We will now make one more change of variables:
  We can apply Lemma~\ref{lem:inverse} to the two-dimensional map
  \[
  (X_1,X_3) \mapsto (X_1',X_3')\equiv ( X_1, -\rho(X_1) \sin\alpha) +
  X_3(0,\cos\alpha),
  \]
  because $\cos\alpha \ne 0$. As a result, this change of variable is
  locally one-one, so $X_3$ can be expressed as a Lipschitz function
  of $X_1'$ and $X_3'$, namely $X_3 \equiv X_3(X_1',X_3')$.  This
  means that in the $\{e_1',e_2',e_3'\}$-coordinate system, the
  surface $\d\Sigma_{-t,t}^{1}$ near $p$ takes the form
  \[
  ( X_1',\;
  \rho(X_1') \cos\alpha + X_3(X_1',X_3') \,\sin\alpha, \; X_3').
  \]
  Since $e_2'=\wh$, the Lipschtizness of the second component above
  shows that the surface is a Lipschitz hypograph in the
  $\wh$-direction. This completes Step~4.

  {\em Finally}, we conclude the proof of the lemma by noting that we
  have verified the condition in Step~1, so $D_t^e$ is a Lipschitz
  hypograph near all the points $p\in \Pi$.
\end{proof}

\begin{proof}[Proof of Theorem~\ref{thm:expand}]
  Define $\om^e_t=\{p+s\hat{v}(p):\; p\in\Gamma,\; s\in (0,t)\},$
  cf.~\eqref{eq:6}. Let $t_0$ be as given by
  Lemma~\ref{lem:bulge}. Set $\om^e = \om_t^e$ for some $0<t<t_0$. Then
  by Lemma~\ref{lem:bulge}, $\om^e$ is Lipschitz.
  
  That $\tilde\om$ is also Lipschitz follows by the same
  techniques. The only points that require any further explanation are
  the points $p \in \Pi$. Consider the domain obtained by
  transporting $\gm_2=\d\om \setminus \gm$, namely
  $\om_{2,t}^{e} =\{q+s\hat{v}(q):\; q\in\gm_2,\; s\in
  (0,t)\}.$ By Lemma~\ref{lem:bulge}, there exists a $t_0$ such that
  for all $0<t<t_0$, its boundary $\smash[b]{\d \om_{2,t}^{e}}$ is Lipschitz.
  Now, observe that the surface $\d \tilde \om$ coincides with the 
  boundary of $\smash[b]{\d \om_{2,t}^{e}}$ (for some value of~$t$) on a small
  neighborhood of~$p\in \Pi$. Hence $\d\tilde \om$ is a Lipschitz
  hypograph near such~$p$.
\end{proof}

\end{document}